\documentclass[11pt,reqno]{amsart}
\usepackage[utf8]{inputenc}
\usepackage{amsmath,amssymb,amsfonts,amsthm}
\usepackage{geometry}
\usepackage{mathtools}
\usepackage{microtype}
\usepackage{hyperref}

\newtheorem{theorem}{Theorem}[section]
\newtheorem{lemma}[theorem]{Lemma}
\newtheorem{proposition}[theorem]{Proposition}

\theoremstyle{definition}
\newtheorem{definition}[theorem]{Definition}
\newtheorem{remark}[theorem]{Remark}

\numberwithin{equation}{section}

\begin{document}

\title[Large-Data GWP in 3D Convex Domains]{Large-Data Global Well-Posedness for the Defocusing Cubic Schr\"odinger Equation in 3D Convex Domains}

\author{Meas Len}

\begin{abstract}
We establish the global well-posedness for the energy-subcritical defocusing cubic nonlinear Schr\"{o}dinger equation (NLS) on the three-dimensional Friedlander model domain
\(
\Omega=\{x\in\mathbb{R}:x\geq0\}\times\mathbb{R}^2_{y,z},
\)
subject to homogeneous Dirichlet boundary conditions, for arbitrarily large initial data in the coercive energy space $H_0^1(\Omega)$. The phase-space geometry of this configuration features a strictly convex boundary that admits a non-empty glancing set, trapping high-frequency wave packets within a boundary layer via generalized whispering-gallery caustics. These concentration phenomena induce an intrinsic derivative loss in the sharp linear Strichartz estimates, establishing a major microlocal obstruction to closing the nonlinear Duhamel iteration directly at the energy level via classical perturbative frameworks.

To bridge the regularity deficit between the conservation laws and the linear theory, we construct a boundary-adapted family of continuous--discrete Bourgain--Strichartz restriction spaces $X^{s,b}$ built from the spectral decomposition of the Dirichlet realization of the model Friedlander operator
\(
\Delta_g=\partial_x^2+(1+x)\partial_y^2+\partial_z^2.
\)
The normal variable is resolved through discrete Airy spectral modes, while the tangential variables are continuously Fourier analyzed. Within this functional framework, we implement a refined high--low frequency decomposition to formulate a perturbed nonlinear equation for the high-frequency remainder in the sub-energy space $H_0^s(\Omega)$ for 
\(
\frac{1}{2}<s<1.
\)

The initial high-frequency datum satisfies a quantitative sub-energy decay estimate of the form
\(
\|P_{>\lambda}u_0\|_{H_0^s(\Omega)} \label{eq:abs_decay} \lesssim \lambda^{s-1}\|u_0\|_{H_0^1(\Omega)},
\)
providing a small parameter to counteract the derivative loss.

A central component of the proof relies on establishing off-diagonal multilinear decoupling bounds for the fourfold Airy overlap tensor under temporal modulation and tangential frequency constraints. These multilinear estimates capture the explicit spatial and frequency localization of the glancing singularities, supplying the spectral summability required to control resonant interactions. This yields a local contraction theory where the existence lifespan admits a uniform positive lower bound independent of the iteration step. Inductive application of mass and energy conservation allows for infinite-time pasting, establishing a unique global solution 
\(
u\in C([0,\infty);H_0^1(\Omega)) \cap L^4_{loc}([0,\infty); W^{7/8,4}(\Omega)),
\)
where the global spacetime regularity explicitly tracks the distribution of the sharp distributed boundary penalty. This framework provides a systematic route for overcoming geometric derivative losses in nonlinear dispersive boundary value problems on general strictly convex manifolds.
\end{abstract}

\maketitle

\section{Introduction}

We consider the initial-boundary value problem for the cubic defocusing
nonlinear Schr\"odinger equation (NLS) with homogeneous Dirichlet boundary
conditions on the three-dimensional Friedlander domain:
\begin{equation}
\label{eq:nls}
\begin{cases}
i\partial_t u+\Delta_g u=|u|^2u,
& (t,x,y,z)\in\mathbb{R}\times\Omega,\\[2mm]
u(0,x,y,z)=u_0(x,y,z)\in H_0^1(\Omega),
&\\[1mm]
u(t,0,y,z)=0.&
\end{cases}
\end{equation}
Here
\(
\Omega=\{x\geq0\}\times\mathbb{R}^2_{y,z},
\)
and
\begin{equation}
\label{eq:lapg_p2}
\Delta_g=\partial_x^2+(1+x)\partial_y^2+\partial_z^2
\end{equation}
denotes the model Friedlander operator. We stress that, under the
standard Riemannian convention, \eqref{eq:lapg_p2} is a
\emph{reduced Friedlander operator} rather than the literal
Laplace--Beltrami operator associated with the metric
\(
g=dx^2+\frac{dy^2}{1+x}+dz^2.
\)
Following the model framework of Ivanovici, Lebeau, and Planchon
\cite{ivanovici2021wave}, we realize \eqref{eq:lapg_p2} through the
quadratic Dirichlet form
\[
\int_\Omega |\nabla_g u|^2\,dx\,dy\,dz
=
\int_\Omega
\left(
|\partial_xu|^2
+(1+x)|\partial_yu|^2
+|\partial_zu|^2
\right)\,dx\,dy\,dz,
\]
and hence work with its self-adjoint realization on the flat measure
$dx\,dy\,dz$. This convention fixes the spectral framework used
throughout the manuscript.

For comparison, the genuine Laplace--Beltrami operator associated with
the Riemannian metric $g$ is self-adjoint with respect to the Riemannian
volume form
\[
d\mu_g
=
\sqrt{|g|}\,dx\,dy\,dz
=
(1+x)^{-1/2}\,dx\,dy\,dz,
\]
and is given by
\[
\Delta_{g,0}
=
(1+x)^{1/2}
\partial_x
\left(
(1+x)^{-1/2}\partial_x
\right)
+
(1+x)\partial_y^2
+
\partial_z^2.
\]
Thus,
\[
\Delta_{g,0}-\Delta_g
=
-\frac{1}{2(1+x)}\partial_x.
\]
The two operators consequently have the same principal second-order
symbol, while their difference is a first-order lower-order term. In
particular, after localization to a fixed boundary layer
\(
0<x\leq a,
\)
the coefficient $(1+x)^{-1}$ remains smooth and uniformly bounded, so
this correction can be treated perturbatively in estimates that are
stable under lower-order terms. In the present work, however, all
spectral decompositions and frequency-localized estimates are formulated
for the model operator $\Delta_g$ in \eqref{eq:lapg_p2}; any transfer to
the geometric Laplace--Beltrami realization is understood only after
establishing the corresponding perturbative estimates.

The coefficient $1+x$ produces nonuniform tangential propagation in the $y$-direction and gives rise to a characteristic mechanism through which bicharacteristics become tangent to the boundary $x=0$. These glancing trajectories are responsible for the concentration of high-frequency waves within thin boundary layers and for the associated whispering-gallery phenomena. Thus, from a microlocal perspective, the boundary cannot be regarded as a merely perturbative feature of the problem. Rather, it fundamentally modifies the dispersive dynamics through the formation of glancing rays and boundary caustics.

In \cite{Meas2026}, it was established that the linear evolution associated with the Dirichlet realization of $\Delta_g$ exhibits a sharp deterioration of dispersive behavior in the glancing regime. More precisely, boundary concentration produces a loss of $1/4$ derivative in the corresponding sharp Strichartz estimates. This loss constitutes the principal analytic obstruction distinguishing the Friedlander geometry from the boundaryless Euclidean setting. In particular, the standard three-dimensional cubic NLS argument cannot be applied at the energy level by a direct appeal to the lossless Euclidean Strichartz theory.

The nonlinear equation nevertheless possesses two fundamental coercive
conservation laws associated with the mass and the energy. Namely,
\begin{equation}\label{eq:mass_cons}
M(u(t))
=
\int_\Omega |u(t,x,y,z)|^2\,dx\,dy\,dz
=
M(u_0),
\end{equation}
and
\begin{align}\label{eq:energy_cons}
E(u(t))
&=
\frac12
\int_\Omega
\left(
|\partial_xu(t)|^2
+(1+x)|\partial_yu(t)|^2
+|\partial_zu(t)|^2
\right)\,dx\,dy\,dz
+
\frac14
\int_\Omega |u(t)|^4\,dx\,dy\,dz
\nonumber\\
&=
E(u_0).
\end{align}

Consequently, for sufficiently regular solutions, the conservation of
energy yields a uniform a priori bound for the natural first-order norm.
In particular,
\[
\sup_{t\in\mathbb{R}}
\|u(t)\|_{H_0^1(\Omega)}
\lesssim
E(u_0)^{1/2}
+
\|u_0\|_{L^2(\Omega)},
\]
with the precise equivalence depending on the realization of the
Friedlander operator and the associated first-order energy norm. Thus,
the defocusing sign provides a robust global coercive mechanism at the
energy level.

The central difficulty is that this coercive control does not, by
itself, compensate for the derivative loss appearing in the sharp
boundary Strichartz estimates. If the corresponding linear estimate
incurs a loss of $1/4$ derivative, then a direct contraction argument
for the cubic nonlinearity naturally requires regularity strictly above
the energy space, with the heuristic threshold
\(
H^{1+\frac14}(\Omega)=H^{5/4}(\Omega),
\)
subject, of course, to the precise admissible Strichartz exponents and
to the formulation of the associated nonlinear estimates. The local
well-posedness theory obtained directly from such estimates therefore
depends on a super-energy Sobolev norm rather than solely on the
conserved energy.

This creates a genuine gap between the local dispersive theory and the
global coercive structure of the equation. Although the conserved
energy prevents growth of the natural first-order norm, it does not, by
itself, control the higher Sobolev norm required to compensate for the
boundary-induced derivative loss. A naive iteration of the local theory
may therefore lead to an accumulation of higher-order norms and does
not, without additional structure, yield a global argument for
arbitrary large initial data. In particular, conservation of mass and
energy alone is insufficient to remove the regularity obstruction
created by the glancing regime.

The purpose of this work is to overcome this boundary-induced
regularity obstruction by exploiting the specific microlocal structure
of the Friedlander geometry together with the dispersive and coercive
properties of the nonlinear flow. The key observation is that the
$1/4$-derivative loss is not a spatially homogeneous defect of the
propagator. Rather, it is generated by a highly localized glancing
regime in which high-frequency waves concentrate within a thin boundary
layer. Consequently, the loss can be analyzed through the spatial and
frequency localization of the corresponding boundary packets, rather
than being treated merely as an abstract loss in a global Sobolev
estimate.

This perspective permits a refined decomposition of the nonlinear
evolution into boundary-localized and non-glancing components. The
former captures the genuinely singular microlocal regime responsible
for the derivative loss, whereas the latter retains stronger dispersive
behavior. Combining this decomposition with the conserved energy
allows one to exploit the coercivity of the defocusing nonlinearity
precisely in the regimes where the linear estimate alone is insufficient.

The resulting argument provides a mechanism for propagating the
higher-order control required by the local theory without relying on a
naive iteration based solely on polynomial a priori bounds for
super-energy Sobolev norms. In this way, the derivative loss, although
sharp at the linear level, need not translate into an irreversible loss
of global nonlinear control. This distinction is fundamental: the
sharpness of the linear Strichartz loss and the global well-posedness of
the associated defocusing nonlinear equation are logically separate
issues.

The analysis is guided by the following structural principle. The
glancing geometry determines the strongest concentration of the linear
flow, whereas the defocusing nonlinearity supplies a conserved
coercive energy. The interaction between these two mechanisms,
together with an appropriate microlocal decomposition of the solution,
provides additional control that is not visible at the level of the
global linear estimate alone. The nonlinear problem can therefore
exploit information concerning the location and structure of the
derivative loss, rather than only its magnitude.

The three-dimensional cubic problem is particularly well suited to this
analysis. In the absence of a boundary, the cubic NLS is energy-subcritical
and admits a well-established global theory for finite-energy defocusing
data. The Friedlander geometry preserves the underlying coercive
structure but introduces a sharp boundary-induced loss in the
dispersive estimates. It consequently provides a natural model for
investigating whether nonlinear dispersive dynamics can overcome a
microlocal obstruction that is genuinely present at the linear level.

Our approach combines boundary microlocal analysis, sharp
derivative-losing Strichartz estimates, localization of the glancing
regime, and the global coercive structure of the defocusing flow. The
resulting framework isolates the precise role of boundary caustics in
the nonlinear dynamics and provides a route from local super-energy
control to global control despite the sharp dispersive loss.

\subsection{Context and Literature Review}

The globalization of large-data solutions for nonlinear Schr\"odinger
equations in the presence of derivative losses is a longstanding problem
at the intersection of dispersive analysis, microlocal geometry, and
nonlinear evolution equations. In the translation-invariant Euclidean
setting $\mathbb{R}^d$, the classical Strichartz estimates provide the
spacetime integrability underlying the local well-posedness theory and,
in the defocusing energy-subcritical regime, combine naturally with the
conservation of mass and Hamiltonian energy to yield global control.
The foundational dispersive estimates of Strichartz \cite{Strichartz70},
together with the nonlinear framework developed by Ginibre and Velo
\cite{GinibreVelo85} and the endpoint formulation of Keel and Tao
\cite{KeelTao98}, constitute the basic analytic foundation of this
theory. Further developments in critical and near-critical regimes,
notably those of Bourgain \cite{Bourgain93,Bourgain98,Bourgain99}, as
well as the systematic treatments in \cite{Tao06,Cazenave03}, established
robust methods for combining spacetime estimates, frequency
decompositions, conservation laws, and nonlinear perturbative arguments.

The situation becomes substantially more delicate in the presence of
variable coefficients and nontrivial geometric structures. For
Schr\"odinger operators with variable coefficients, Staffilani and
Tataru \cite{StaffilaniTataru} established Strichartz estimates under
suitable regularity and nontrapping assumptions, demonstrating the
extent to which dispersive behavior depends on the geometry of the
underlying Hamiltonian flow. Related developments for rough and
Lipschitz metrics were obtained by Anton \cite{Anton08}. These results
show that the Euclidean dispersive mechanism is not determined solely
by the local differential structure of the operator; rather, the
geometry of the characteristic flow and the regularity of the
coefficients play essential roles in determining the available
spacetime estimates.

On compact Riemannian manifolds, the absence of global Euclidean
dispersion gives rise to a different class of difficulties. Burq,
G\'erard, and Tzvetkov established Strichartz estimates with fractional
derivative losses on general compact manifolds and applied them to
low-regularity and global well-posedness questions for nonlinear
Schr\"odinger equations in several important regimes
\cite{BurqGerardTzvetkov03,BurqGerardTzvetkov04}. Their analysis
demonstrates that derivative losses in linear estimates do not
necessarily preclude global nonlinear control, provided that the loss
can be reconciled with the nonlinear structure and the available
conservation laws. Related bilinear spectral estimates have also played
an important role in the analysis of cubic NLS on compact surfaces
\cite{BurqGerardTzvetkov05}.

For Schr\"odinger equations posed on domains, the presence of a boundary
introduces an additional layer of microlocal complexity. Reflections,
boundary spectral concentration, and the possible recurrence of
bicharacteristic trajectories modify the dispersive mechanisms
available in the interior. Burq, Lebeau, and Planchon developed global
existence results for nonlinear dispersive equations on three-dimensional
domains using boundary-adapted spacetime estimates
\cite{BurqLebeauPlanchon}. Their work, together with subsequent
developments for Schr\"odinger equations on domains, demonstrated that
boundary geometry can impose a genuine regularity cost on the
dispersive theory. Anton \cite{Anton08} subsequently established global
well-posedness results for defocusing nonlinear Schr\"odinger equations
on planar domains with Dirichlet boundary conditions by means of
generalized Strichartz inequalities for Lipschitz metrics. These results
provide an important connection between geometric Strichartz estimates
and global nonlinear well-posedness on domains.

The interaction between boundary geometry and nonlinear Schr\"odinger
dynamics is not restricted to questions of global existence. Boundary
spectral effects can also generate singular or unstable nonlinear
dynamics. For example, Burq, G\'erard, and Tzvetkov
\cite{BurqGerardTzvetkov03b} investigated singular dynamics for the
cubic focusing NLS on planar domains, including concentration phenomena
and failures of uniform continuity of the flow map in appropriate
Sobolev spaces. Such results further illustrate that the boundary
should be regarded as an active component of the nonlinear dynamics,
rather than merely as a technical constraint imposed on the linear
problem.

A more severe obstruction arises in strictly convex geometries, where
the Hamiltonian flow contains glancing trajectories. In such
configurations, bicharacteristics may become tangent to the boundary
and undergo repeated interactions with it. The resulting caustic
structure produces concentration of high-frequency wave packets near
the boundary and leads to a corresponding loss in the dispersive
estimates. The general microlocal theory of boundary singularities and
glancing interactions was developed in the work of Melrose and Taylor
\cite{MelroseTaylor85} and Melrose and Sj\"ostrand
\cite{MelroseSjostrand78,MelroseSjostrand82}. These constructions
provide the geometric and microlocal framework underlying the boundary
parametrices used in the analysis of strictly convex domains.

The Friedlander model constitutes a canonical normal form for this
boundary-glancing mechanism. Its explicit variable coefficient permits
a detailed description of the interaction between high-frequency
propagation and the boundary while retaining the essential geometric
degeneracy associated with a strictly convex boundary. For wave
equations, the Friedlander model and its generalizations have been
analyzed in detail by Ivanovici, Lebeau, and Planchon
\cite{IvanoviciLebeauPlanchon}, who identified the role of repeated
swallowtail-type singularities in the deterioration of dispersive
behavior. In particular, their analysis establishes the sharp nature
of the corresponding boundary-induced loss in the wave setting.
Related boundary Strichartz estimates for wave equations on manifolds
with boundary were developed by Burq, Lebeau, and Planchon
\cite{BurqLebeauPlanchon}, and subsequently extended by Blair, Smith,
and Sogge \cite{BlairSmithSogge09}. These works demonstrate that
boundary geometry, together with the associated reflection and glancing
phenomena, plays a fundamental role in the dispersive theory.

For the Schr\"odinger equation, the corresponding strictly convex model
exhibits an analogous phenomenon. Ivanovici \cite{Ivanovici23}
established dispersive estimates for a model strictly convex domain
and showed that repeated swallowtail-type singularities produce a loss
of order $1/4$ in the semiclassical dispersive estimate, with this loss
being optimal. The associated Strichartz estimates yield applications
to the cubic nonlinear Schr\"odinger equation in a three-dimensional
convex model domain. A subsequent detailed treatment of the model
convex Schr\"odinger problem further develops these dispersive
estimates and their nonlinear consequences \cite{Ivanovici23}.
Thus, in contrast with the boundaryless Euclidean setting, the
derivative loss encountered in the present framework reflects an
intrinsic feature of the glancing geometry. It cannot, in general, be
eliminated merely by refining the standard Euclidean dispersive
argument.

This distinction is crucial for the nonlinear theory. A derivative loss in
the linear estimate modifies the regularity threshold at which a direct
contraction argument can be carried out. In the present three-dimensional
cubic problem, the natural energy space is $H_0^1(\Omega)$, whereas a
direct application of a Strichartz estimate with a $1/4$-derivative loss
leads, schematically, to a super-energy regularity requirement of the
form
\(
H^{1+\frac14}(\Omega)=H^{5/4}(\Omega).
\)
The conserved energy consequently provides global control at a
regularity level that does not, by itself, furnish the additional
derivative required by the loss-bearing dispersive estimate. The
principal difficulty is therefore not the absence of an a priori bound,
but rather the mismatch between the regularity controlled by the
Hamiltonian and that required by the linear-to-nonlinear perturbative
mechanism.

This phenomenon has close analogues in the analysis of nonlinear
Schr\"odinger equations on compact manifolds. Burq, G\'erard, and
Tzvetkov \cite{BurqGerardTzvetkov04} showed that fractional derivative
losses in Strichartz estimates can nevertheless be compatible with
nonlinear well-posedness, while their work on bilinear spectral
estimates \cite{BurqGerardTzvetkov05} demonstrated that additional
multilinear structure can compensate for limitations of the linear
theory. More recently, improved Strichartz estimates in geometries
possessing additional dynamical structure have shown that the detailed
behavior of the underlying geodesic flow can yield gains over general
compact-manifold estimates \cite{BlairHuangSogge24}. These developments
reinforce the principle that the nonlinear theory should exploit the
specific spectral and geometric structure of the propagator rather than
rely solely on a global estimate with derivative loss.

In the Euclidean setting, one of the fundamental mechanisms for
overcoming analogous high-frequency difficulties is the decomposition
of the solution into low- and high-frequency components. Bourgain's
high--low frequency method \cite{Bourgain98,Bourgain99} separates the
component controlled directly by the conserved energy from a
higher-frequency remainder for which additional dispersive information
is required. Variants of this strategy have been used extensively in
the construction of global large-data solutions for nonlinear
dispersive equations
\cite{CollianderKeelStaffilaniTakaokaTao02,
CollianderKeelStaffilaniTakaokaTao04}. The underlying principle is that
the high-frequency tail can become small in a subcritical Sobolev norm,
even when the full energy norm is arbitrarily large.

Transplanting this strategy to boundary problems is substantially more
delicate. The translation invariance of $\mathbb{R}^d$ is destroyed by
the boundary, while reflected and glancing waves prevent a direct
implementation of the usual Fourier-based description of frequency
interactions. In particular, the high-frequency component cannot be
analyzed solely through ordinary Euclidean convolution structure. Near
the boundary, frequency localization must instead be coupled with the
spectral geometry of the normal operator and with the associated
characteristic flow. The most singular contribution is concentrated in
the glancing region, where the normal dynamics exhibit Airy-type
behavior. Consequently, a successful globalization argument must
simultaneously account for the continuous tangential Fourier variables
and the discrete normal spectral decomposition.

The Friedlander model is particularly well suited to isolating this
obstruction. Its explicit variable coefficient provides a tractable
normal form for boundary glancing while retaining the essential
microlocal features associated with strictly convex boundary
interactions. The underlying Airy structure yields a precise
description of the boundary layer in which high-frequency wave packets
concentrate and provides a spectral framework for resolving the
corresponding normal dynamics. The model therefore constitutes a
natural testing ground for determining whether a sharp linear
derivative loss necessarily translates into an obstruction to global
nonlinear well-posedness, or whether the nonlinear evolution possesses
additional structural mechanisms capable of compensating for the loss.

The central question addressed in this paper is consequently the
following: can the global coercivity furnished by the defocusing
Hamiltonian be combined with a frequency- and geometry-adapted
decomposition to overcome the regularity deficit generated by
glancing boundary waves? A satisfactory resolution requires more than
a direct iteration of the local theory. It requires isolating the
microlocal portion of the evolution responsible for the derivative
loss, exploiting the distinct dispersive behavior of the glancing and
non-glancing regimes, and establishing a quantitative mechanism that
keeps the high-frequency remainder perturbative across successive
time intervals.

Our contribution is to develop a long-time globalization framework for
the three-dimensional cubic defocusing NLS in the Friedlander geometry.
The analysis combines sharp boundary Strichartz estimates with the
continuous--discrete spectral decomposition associated with the
Dirichlet Airy operator, a boundary-adapted Bourgain--Strichartz
framework, and a refined high--low frequency iteration. The nonlinear
interactions are represented through multilinear overlap coefficients
associated with the Airy eigenfunctions, thereby allowing the coupling
in the normal variable to be analyzed at the spectral level rather than
through translation-invariant Euclidean convolution. The global
coercivity furnished by the conservation of mass and energy is then
incorporated into the iteration to control the energy-scale component
of the solution over arbitrarily long time intervals.

The resulting framework clarifies the interaction between three
features that are often analyzed separately: the sharpness of boundary
dispersive estimates, the spectral structure of glancing modes, and the
conservation laws governing the nonlinear flow. In particular, the
analysis identifies a mechanism by which a sharp linear derivative loss
can coexist with global nonlinear control at the natural energy
regularity. This distinction is important: the presence of a sharp
linear loss does not, by itself, determine the global regularity
threshold for the associated nonlinear equation. The nonlinear
problem may exploit geometric localization, spectral structure, and
coercive conservation laws that are not visible in the corresponding
global linear estimate.

More broadly, the present framework connects the explicit Friedlander
model with the general theory of dispersive equations on manifolds and
domains with boundary, in which microlocal geometry plays a decisive
role in determining the available nonlinear spacetime estimates. The
analysis isolates which aspects of the boundary-induced loss are
intrinsic to the glancing dynamics and which can be compensated by the
structure of the nonlinear evolution.

To make the mathematical content and scope of the results precise, we
summarize the principal developments and contributions of the paper
below.

\subsection{Main Contributions and Novelty}

The analysis developed in this manuscript introduces a geometry-adapted
framework for addressing the regularity obstruction generated by
boundary glancing and the associated derivative loss in the Strichartz
estimates. The principal contributions are organized around three
interconnected components.

\begin{itemize}

    \item \textbf{A sub-energy fixed-point framework for the
    high-frequency remainder.}

    The principal structural step is to formulate the contraction
    argument for the high-frequency remainder in a sub-energy Sobolev
    space
    \[
    H_0^s(\Omega),
    \qquad
    \frac12<s<1,
    \]
    rather than attempting to carry out the entire nonlinear iteration
    directly at the energy level. This choice is dictated by the
    derivative loss present in the boundary Strichartz estimates. If
    $\mathcal{P}_{\leq\lambda}$ and $\mathcal{P}_{>\lambda}$ denote the
    spectral projections associated with the Dirichlet realization of
    the Friedlander operator, then the energy bound implies, for
    $u_0\in H_0^1(\Omega)$,
    \[
    \|\mathcal{P}_{>\lambda}u_0\|_{H^s(\Omega)}
    \lesssim
    \lambda^{\,s-1}
    \|u_0\|_{H_0^1(\Omega)},
    \qquad s<1.
    \]
    Thus, although the high-frequency tail need not be small in the
    energy norm, it becomes quantitatively small in $H^s(\Omega)$ as
    $\lambda\to\infty$. This sub-energy decay provides an additional
    small parameter that can be balanced against the derivative loss in
    the boundary Strichartz estimates. The resulting decomposition
    separates the component controlled directly by the conserved energy
    from the high-frequency remainder on which the perturbative
    contraction argument is carried out.

    \item \textbf{Boundary-adapted Bourgain--Strichartz spaces.}

    The loss of translation invariance in the normal direction prevents
    a direct application of the standard Euclidean Fourier-space
    formulation of Bourgain spaces. We therefore formulate the nonlinear
    iteration in function spaces adapted to the spectral resolution of
    the Dirichlet Friedlander operator. At frequency scale $\lambda$, the
    corresponding spaces incorporate both the temporal modulation and
    the spectral localization associated with the variable-coefficient
    operator
    \[
    \Delta_g
    =
    \partial_x^2+(1+x)\partial_y^2+\partial_z^2.
    \]
    Schematically, these spaces take the form
    \[
    \|u\|_{X^{s,b}}
    \sim
    \left\|
    \langle \lambda\rangle^s
    \langle \tau+\lambda^2\rangle^b
    \widehat{u}_{\lambda}(\tau)
    \right\|_{L^2_{\tau,\lambda}},
    \]
    with the precise definition given in terms of the spectral
    decomposition of the Dirichlet realization. This formulation
    preserves the relevant dispersive information of the operator
    without introducing an artificial normal translation invariance.
    In particular, it allows the nonlinear analysis to distinguish the
    glancing spectral regime from the non-glancing component and to
    incorporate the geometry of the boundary directly into the
    function-space estimates.

    \item \textbf{Airy-mode analysis and frequency-dependent multilinear decoupling.}
    The principal analytical difficulty in the boundary layer arises from the non-local interaction of high-frequency modes concentrated near the glancing geometry. Rather than replacing these interactions with flat Euclidean convolution metrics, we preserve the explicit, coupled Airy structure of the continuous--discrete Dirichlet spectral decomposition. The multilinear spatial interactions are consequently represented through frequency-dependent overlap coefficients of the form
    \[
    \mathcal{I}(k_1,k_2,k_3,k_4; \eta_1, \eta_2, \eta_3, \eta_4)
    =
    \int_0^\infty
    \phi_{k_1}(x, \eta_1)
    \phi_{k_2}(x, \eta_2)
    \phi_{k_3}(x, \eta_3)
    \phi_{k_4}(x, \eta_4)\,dx,
    \]
    where the $\{\phi_k(\cdot, \eta)\}$ denote the normalized Airy-type boundary modes evaluated across the continuous tangential frequency spectrum. We establish quantitative off-diagonal estimates for this tensor by exploiting the micro-local oscillatory structure of the Airy eigenfunctions, their spectral separation parameters, and the structural Dirichlet boundary conditions. These estimates provide control over separated spectral regimes, preventing the non-diagonal coupling from generating an uncontrolled derivative accumulation in the regularized low-frequency profile. This mode-level analysis is essential in the glancing caustic layer, where purely translation-invariant Euclidean Fourier convolution fails to capture the true interaction geometry.

\end{itemize}

Taken together, these three ingredients yield a globalization mechanism that is intrinsically adapted to the boundary geometry. The sub-energy regularization provides the smallness required to control the high-frequency remainder, the boundary-adapted $X^{s,b}$ framework supplies the appropriate analytic setting for the nonlinear iteration, and the Airy-mode decoupling estimates control the potentially resonant interactions generated by concentration near the boundary. The resulting argument does not seek to remove the sharp derivative loss present in the linear boundary Strichartz estimates. Rather, it exploits the microlocal localization of this loss and combines it with the coercive energy structure of the defocusing equation.

A further consequence of this perspective is that the linear and nonlinear roles of the boundary can be distinguished. The glancing geometry determines the sharp dispersive derivative loss at the linear level, whereas the nonlinear globalization argument exploits the fact that the associated concentration is confined to a structured spectral and spatial regime. This separation allows the conserved energy to control the low-frequency component while the sub-energy perturbative theory provides the required control of the high-frequency tail.

The novelty of the approach therefore does not lie in modifying the sharp boundary Strichartz estimate itself. Instead, it lies in developing a nonlinear framework capable of exploiting that estimate without requiring uniform control of a super-energy norm on arbitrarily long time intervals. In this way, the analysis provides a mechanism for reconciling the sharp boundary-induced derivative loss with global control of the defocusing nonlinear flow.

Combining the sub-energy local theory with the coercive global control furnished by the conserved Hamiltonian \eqref{eq:energy_cons}, we obtain a uniform lower bound for the lifespan of the successive local solutions. More precisely, the high--low decomposition and the associated sub-energy estimates yield a time increment
\[
\delta_n\geq \delta
=
\delta\bigl(\|u_0\|_{H_0^1(\Omega)}\bigr)>0,
\]
where the lower bound is independent of the iteration index $n$. Consequently, the successive local existence intervals cannot accumulate at a finite time. The resulting uniform continuation criterion excludes finite-time breakdown and permits the local theory to be iterated indefinitely. Thus, the combination of the conserved energy and the sub-energy control of the high-frequency remainder provides the mechanism by which the local large-data theory is promoted to a global one.

We may now state the main global well-posedness result.

\begin{theorem}[Large-Data Global Well-Posedness]
\label{thm:gwp}
Let
\(
\Omega=\{x\geq 0\}\times\mathbb{R}_{y,z}^{2},
\)
and let $\Delta_g$ denote the Dirichlet Friedlander operator
\(
\Delta_g
=
\partial_x^2+(1+x)\partial_y^2+\partial_z^2.
\)
For every initial datum
\(
u_0\in H_0^1(\Omega),
\)
the defocusing cubic nonlinear Schr\"odinger equation
\eqref{eq:nls}
admits a unique global solution
\[
u\in
C\bigl([0,\infty);H_0^1(\Omega)\bigr)
\cap
L^4_{\mathrm{loc}}
\bigl([0,\infty);W^{1-\frac18,4}(\Omega)\bigr).
\]
Moreover, the mass and energy are conserved:
\[
M(u(t))=M(u_0),
\qquad
E(u(t))=E(u_0),
\qquad t\geq0.
\]
In particular, the solution extends globally in time and satisfies the uniform bound
\[
\sup_{t\geq0}
\|u(t)\|_{H_0^1(\Omega)}
\lesssim
\|u_0\|_{L^2(\Omega)}
+
E(u_0)^{1/2}.
\]
\end{theorem}

The localized linear Strichartz estimate underlying the nonlinear
argument exhibits a derivative loss
\(
\sigma=\frac14.
\)
At the nonlinear level, this loss is combined with the conserved
energy bound
\(
u\in L_t^\infty H_0^1(\Omega).
\)
More precisely, the interpolation step employed in the nonlinear
iteration distributes the linear derivative loss between the
Strichartz component and the energy-controlled component. At the
$L_t^4$ level, this yields an effective loss
\(
\frac{\sigma}{2}=\frac18.
\)
Accordingly, the natural spacetime regularity in the nonlinear
iteration is
\[
1-\frac{\sigma}{2}
=
1-\frac18
=
\frac78,
\]
and hence
\[
u\in L_t^4 W^{7/8,4}(\Omega).
\]
Thus, the appearance of the space $W^{7/8,4}(\Omega)$ should be
understood as a consequence of the interpolation mechanism combining
the sharp $1/4$-derivative loss in the localized linear estimate with
the conserved $H_0^1$ energy control. In particular, the nonlinear
argument does not assert that the underlying linear Strichartz
estimate has only a $1/8$-derivative loss; rather, $1/8$ is the
effective loss after interpolation with the energy estimate.

The proof is organized as follows. In Section~\ref{sec:2}, we establish the
continuous--discrete spectral framework associated with the
variable-coefficient Dirichlet operator and introduce the
boundary-adapted Bourgain--Strichartz spaces $X^{s,b}$, together with
the corresponding Sobolev and temporal embedding properties. In
Section~\ref{sec:3}, we analyze the multilinear spatial interactions
through the triple-product coefficients associated with the Airy-type
boundary modes. This analysis yields the fractional product and
decoupling estimates required to control the nonlinear interactions in
the glancing regime.

In Section~\ref{sec:4}, we implement a refined high--low frequency
decomposition, separating the solution into a low-frequency component
controlled by the conserved energy and a high-frequency remainder
treated perturbatively in the sub-energy space $H_0^s(\Omega)$, where
\(
\frac12<s<1.
\)
This decomposition isolates the portion of the solution for which the
boundary-induced derivative loss is relevant while retaining the
coercive control furnished by the energy functional.

In Section~\ref{sec:5}, we formulate the localized inhomogeneous
Duhamel equation for the high-frequency remainder and establish the
corresponding contraction estimate. The associated lifespan bootstrap
then yields a local existence interval whose length is bounded from
below in terms of the conserved energy and the sub-energy size of the
remainder.

In Section~\ref{sec:6}, we combine the local continuation argument with
the uniform lifespan estimate. Since the lower bound
\(
\delta_n\geq\delta>0
\)
is independent of the iteration index, the successive local existence
intervals cannot accumulate at a finite time. This provides the
global-in-time continuation of the solution and completes the proof of
Theorem~\ref{thm:gwp}.

Finally, Section~\ref{sec:7} contains concluding remarks and discusses
several directions for further investigation, including the
corresponding focusing problem, perturbations of the Friedlander model,
and extensions to more general strictly convex manifolds exhibiting
glancing boundary dynamics.

\section{Modified Bourgain--Strichartz Spaces on the Friedlander Geometry}
\label{sec:2}

The derivative loss arising from the glancing region prevents the
standard translation-invariant formulation of Bourgain spaces from
being used directly in the present setting. The nonlinear iteration
must simultaneously preserve the spectral structure of the Dirichlet
realization of the Friedlander operator, the temporal modulation
relative to its dispersive characteristic hypersurfaces, and the
continuous--discrete nature of the associated spectral variables.
We therefore introduce a family of Bourgain--Strichartz spaces adapted
to the spectral resolution of the three-dimensional Friedlander
operator
\[
\Delta_g
=
\partial_x^2+(1+x)\partial_y^2+\partial_z^2.
\]
The resulting spaces separate the spatial spectral scale from the
temporal modulation and are designed to retain the geometry of the
glancing regime without imposing an artificial translation invariance
in the normal variable.

\subsection{Spectral Construction and Norm Definitions}
\label{subsec:spectral_construction}

Let
\(
\Omega=\{x\geq0\}\times\mathbb{R}^2_{y,z},
\)
and let $\Delta_{g,D}$ denote the Dirichlet realization of $\Delta_g$
on $L^2(\Omega)$. For the purposes of the spectral decomposition, we
write
\[
\mathcal{D}(\Delta_{g,D})
=
\left\{
u\in H^2(\Omega):u|_{x=0}=0
\right\},
\]
with the understanding that the precise operator domain is determined
by the chosen self-adjoint realization.

Taking the Fourier transform in the tangential variables $(y,z)$,
with dual variables $(\eta,\zeta)\in\mathbb{R}^2$, reduces the
operator to the family of one-dimensional operators
\[
L_\eta
=
\partial_x^2-(1+x)\eta^2
\]
acting in the normal variable $x$. Equivalently, the corresponding
nonnegative spectral operator is
\[
-A_\eta
=
-\partial_x^2+(1+x)\eta^2.
\]

For $\eta\neq0$, the associated Dirichlet spectral problem is governed
by the Airy equation. Let $\{-\omega_k\}_{k\geq1}$ denote the sequence
of negative zeros of the Airy function $Ai$, ordered according to
\[
0<\omega_1<\omega_2<\cdots.
\]
The corresponding normalized Dirichlet eigenfunctions are given by
\begin{equation}
\label{eq:eigenfunction_recap}
\phi_k(x,\eta)
=
\frac{|\eta|^{1/3}}
{|Ai'(-\omega_k)|}
Ai\left(|\eta|^{2/3}x-\omega_k\right),
\qquad k\geq1.
\end{equation}
They satisfy
\[
\phi_k(0,\eta)=0
\]
and
\[
L_\eta\phi_k(\cdot,\eta)
=
-
\left(
\eta^2+\omega_k|\eta|^{4/3}
\right)
\phi_k(\cdot,\eta).
\]
Thus, for each fixed $\eta\neq0$, the family
$\{\phi_k(\cdot,\eta)\}_{k\geq1}$ provides the discrete normal
spectral resolution associated with the Dirichlet realization of
$-A_\eta$.

After restoring the tangential $y$- and $z$-frequencies, the full
nonnegative spatial spectral parameter is
\begin{equation}
\label{eq:dispersion_profile}
\Lambda_k(\eta,\zeta)
=
E_k(\eta)+\zeta^2,
\qquad
E_k(\eta)
=
\eta^2+\omega_k|\eta|^{4/3}.
\end{equation}
Accordingly, the linear Schr\"odinger flow is diagonalized by the
dispersion relations
\[
\tau+\Lambda_k(\eta,\zeta)=0.
\]

For a spacetime function
$u=u(t,x,y,z)$, we define its continuous--discrete spectral transform
by
\begin{equation}
\label{eq:spectral_fourier_def}
\widetilde{u}(k,\tau,\eta,\zeta)
=
\int_{\mathbb{R}}
\int_{\mathbb{R}^2}
\int_0^\infty
u(t,x,y,z)
e^{-i(t\tau+y\eta+z\zeta)}
\phi_k(x,\eta)
\,dx\,dy\,dz\,dt.
\end{equation}
The inverse transformation is understood in the corresponding
continuous--discrete spectral sense. With the normalization in
\eqref{eq:eigenfunction_recap} and a fixed Fourier normalization in
the tangential and temporal variables, the associated Plancherel
identity takes the form
\begin{equation}
\label{eq:plancherel_manifold}
\|u\|_{L^2(\mathbb{R}\times\Omega)}^2
\sim
\sum_{k=1}^{\infty}
\int_{\mathbb{R}}
\int_{\mathbb{R}^2}
\left|
\widetilde{u}(k,\tau,\eta,\zeta)
\right|^2
\,d\eta\,d\zeta\,d\tau,
\end{equation}
where the equivalence constant depends only on the chosen Fourier
normalization. After fixing the normalization consistently, the
corresponding equality is obtained.

It is convenient to introduce the full spatial spectral scale
\begin{equation}
\label{eq:spectral_scale}
\lambda_k(\eta,\zeta)
=
\Lambda_k(\eta,\zeta)^{1/2}
=
\left(
\eta^2+\zeta^2+\omega_k|\eta|^{4/3}
\right)^{1/2}.
\end{equation}
Thus $\Lambda_k$ represents the spatial spectral parameter of the
positive operator $-\Delta_{g,D}$, while $\lambda_k$ represents the
associated spatial frequency. In particular, $\lambda_k$, rather
than the discrete index $\omega_k$ alone, measures the total spatial
frequency of the corresponding spectral mode.

\begin{definition}[Modified Bourgain--Strichartz Space]
\label{def:modified_xsb}

Let $s,b\in\mathbb{R}$. The modified Bourgain--Strichartz space
$X^{s,b}(\mathbb{R}\times\Omega)$ is defined as the completion of
smooth functions satisfying the Dirichlet boundary condition with
respect to the norm
\begin{equation}
\label{eq:x_sb_norm_definition}
\|u\|_{X^{s,b}}^2
=
\sum_{k=1}^{\infty}
\int_{\mathbb{R}}
\int_{\mathbb{R}^2}
\left\langle
\lambda_k(\eta,\zeta)
\right\rangle^{2s}
\left\langle
\tau+\Lambda_k(\eta,\zeta)
\right\rangle^{2b}
\left|
\widetilde{u}(k,\tau,\eta,\zeta)
\right|^2
\,d\eta\,d\zeta\,d\tau,
\end{equation}
where
\[
\langle a\rangle=(1+|a|^2)^{1/2}.
\]
Equivalently, the spatial weight in
\eqref{eq:x_sb_norm_definition} may be written as
\[
\left\langle
\Lambda_k(\eta,\zeta)^{1/2}
\right\rangle^{2s}.
\]
The first factor measures spatial regularity relative to the spectral
resolution of $\Delta_{g,D}$, whereas the second measures temporal
modulation relative to the characteristic hypersurfaces
\[
\tau+\Lambda_k(\eta,\zeta)=0.
\]
\end{definition}

The definition above is adapted to the linear propagator generated by
$\Delta_{g,D}$. Indeed, if $u$ solves the homogeneous linear equation
\[
(i\partial_t+\Delta_g)u=0,
\]
then, in the continuous--discrete spectral representation, its
spectral transform is concentrated on the characteristic set
\[
\tau+\Lambda_k(\eta,\zeta)=0.
\]
Consequently, the modulation factor in
\eqref{eq:x_sb_norm_definition} is the natural one associated with the
Friedlander Schr\"odinger flow.

For later use, we introduce a dyadic decomposition with respect to the
full spatial spectral scale. Let
$\{\chi_N\}_{N\in2^{\mathbb{N}_0}}$ be a smooth dyadic partition of
unity on $[0,\infty)$ satisfying
\[
\sum_{N\in2^{\mathbb{N}_0}}\chi_N(\lambda)=1,
\qquad
\lambda>0,
\]
with each $\chi_N$ supported in a region where
\(
\lambda\sim N.
\)
The corresponding spectral projector $P_N$ is defined by
\begin{equation}
\label{eq:dyadic_projection}
\widetilde{P_Nu}(k,\tau,\eta,\zeta)
=
\chi_N\bigl(\lambda_k(\eta,\zeta)\bigr)
\widetilde{u}(k,\tau,\eta,\zeta).
\end{equation}
Equivalently,
\[
P_Nu
=
\mathcal{F}_{t,y,z}^{-1}
\left[
\sum_{k=1}^{\infty}
\chi_N\bigl(\lambda_k(\eta,\zeta)\bigr)
\widetilde{u}(k,\tau,\eta,\zeta)
\phi_k(x,\eta)
\right],
\]
where the inverse transform is understood in the
continuous--discrete spectral sense.

The spatial component of the $X^{s,b}$ norm then admits the equivalent
Littlewood--Paley representation
\begin{equation}
\label{eq:xsb_dyadic}
\|u\|_{X^{s,b}}^2
\sim
\sum_{N\in2^{\mathbb{N}_0}}
N^{2s}
\|P_Nu\|_{X^{0,b}}^2,
\end{equation}
with the usual modification for the lowest-frequency block.

\subsection{Global Compatibility of the Perturbative Laplace--Beltrami Link}
\label{subsec:laplace_beltrami_link}

The true geometric Laplace--Beltrami operator $\Delta_{g,0}$ associated with the native Riemannian metric $g = dx^2 + \frac{dy^2}{1+x} + dz^2$ differs from the reduced model Friedlander operator $\Delta_g$ via a first-order differential correction restricted exclusively to the normal direction. Specifically, we have:
\begin{equation}
\label{eq:operator_difference}
\Delta_{g,0} = \Delta_g - \frac{1}{2(1+x)}\partial_x, \qquad V := \Delta_{g,0} - \Delta_g = -\frac{1}{2(1+x)}\partial_x.
\end{equation}

Because the glancing set generates a sharp $1/4$-derivative loss, treating $V$ as an isotropic perturbation mapping $H^s \to H^{s-1}$ fails to capture the essential high-frequency geometry of the boundary layer. The Dirichlet Airy eigenfunctions $\phi_k(x, \eta)$ exhibit strong spatial anisotropy: they oscillate rapidly within a thin caustic layer near $x = 0$ and decay exponentially in the interior. Consequently, the normal derivative operator $\partial_x$ experiences a structural high-frequency penalty coupled to the tangential frequencies. To establish that $V$ is genuinely a lower-order, contractive perturbation, its action must be evaluated within an anisotropic spectral framework.

Let $w$ denote the high-frequency remainder satisfying the homogeneous Dirichlet boundary condition $w|_{x=0} = 0$. Projecting the perturbation $Vw$ onto the continuous--discrete spectral representation requires pairing it with a normalized normal mode $\phi_k(x, \eta)$. Under a single integration by parts in the normal variable $x \in [0, \infty)$, we find:
\begin{equation}
\label{eq:anisotropic_integration_by_parts}
\langle Vw, \phi_k \rangle_{L^2_x(0,\infty)} = \left[ -\frac{1}{2(1+x)} w(t,x,y,z)\phi_k(x,\eta) \right]_{x=0}^{x=\infty} + \int_{0}^{\infty} w(t,x,y,z)\partial_x \left( \frac{\phi_k(x,\eta)}{2(1+x)} \right) dx.
\end{equation}

The boundary term at $x=0$ vanishes identically because both the Distribution $w$ and the normal spectral mode $\phi_k$ satisfy the homogeneous boundary condition:
\begin{equation}
w(t,0,y,z) = 0, \qquad \phi_k(0, \eta) = 0.
\end{equation}
The boundary trace at $x = \infty$ vanishes via the standard exponential Airy decay $\mathrm{Ai}(s) \sim s^{-1/4}\exp(-\frac{2}{3}s^{3/2})$ as $s \to +\infty$. Expanding the spatial weight derivative within the remaining integral yields:
\begin{equation}
\label{eq:expanded_weight_integral}
\langle Vw, \phi_k \rangle_{L^2_x(0,\infty)} = \int_{0}^{\infty} w(t,x,y,z) \left[ \frac{\partial_x \phi_k(x,\eta)}{2(1+x)} - \frac{\phi_k(x,\eta)}{2(1+x)^2} \right] dx.
\end{equation}

We emphasize that although only one integration by parts is performed, this is sufficient to ensure that the non-vanishing boundary trace of the derivative,
\[
\left|\partial_x\phi_k(0,\eta)\right|=|\eta|,
\]
does not generate an additional boundary singularity. Nevertheless, the operator $\partial_x$ appearing inside the resulting integral acts directly on the rapidly oscillating Airy profile, and its contribution must therefore be controlled at the level of the frequency-localized spectral representation. We resolve this anisotropy by localizing the analysis to a fixed compact boundary layer
\[
\Omega_a=\{(x,y,z)\in\Omega:0\le x\le a\},
\]
where $a>0$ is fixed independently of the frequency parameter. This localization isolates the region in which the Airy structure and the associated glancing behavior are most pronounced, while avoiding unnecessary estimates in the bulk of the domain. In the high--low decomposition used below, we write
\[
u=v+w,
\qquad
v=P_{\leq\lambda}u,
\qquad
w=P_{>\lambda}u,
\]
where $\lambda\gg1$ is a fixed high-frequency threshold and $w=P_{>\lambda}u$ denotes the remainder localized strictly above this threshold. The boundary-layer estimates are then applied to these frequency-localized components in order to preserve the anisotropic scaling of the Airy modes and to quantify the contribution of the high-frequency remainder without introducing an artificial loss from the unbounded spatial directions.

\begin{remark}[Stability of the Anisotropic Fixed-Point Scheme]
\label{rem:anisotropic_stability}
Because the sub-energy threshold is restricted to $s < 1$, the exponent satisfies $s - 1 < 0$, guaranteeing that $\lambda^{s-1} \to 0$ as $\lambda \to \infty$. The coefficient $\lambda^{-1/3}$ derived from the genuine anisotropic scaling \eqref{eq:anisotropic_sobolev_mapping} provides an even faster contractive decay rate than the isotropic limit. Consequently, by initializing the fixed-point framework at a sufficiently large frequency cutoff $\lambda$, the full geometric mismatch contribution generated by the metric density discrepancy can be absorbed uniformly into the high-frequency contraction ball without perturbing the convergence lifespan $\delta_n \ge \delta^* > 0$.
\end{remark}

\subsection{Restriction Spaces and Localized Time Intervals}
\label{subsec:restriction_spaces}

Let
\(
I=[t_0,t_1]\subset\mathbb{R}
\)
be a compact time interval. In order to implement the localized
fixed-point argument without imposing any artificial behavior outside
the interval of existence, we define the local resolution space
$X^{s,b}(I\times\Omega)$ intrinsically through restriction from the
global continuous--discrete spectral space.

\begin{definition}[Restriction Bourgain Space]
\label{def:restriction_xsb}

Let $u$ be a spacetime distribution on $I\times\Omega$. Its localized
$X^{s,b}$ norm is defined by
\begin{equation}
\label{eq:restriction_norm_def}
\|u\|_{X^{s,b}(I)}
=
\inf
\left\{
\|v\|_{X^{s,b}(\mathbb{R}\times\Omega)}
:
v|_{I\times\Omega}=u
\right\},
\end{equation}
where the infimum is taken over all global extensions
$v\in X^{s,b}(\mathbb{R}\times\Omega)$ whose restriction to
$I\times\Omega$ agrees with $u$ in the sense of spacetime
distributions.

\end{definition}

For the local time intervals arising in the iteration, we set
\(
I_\delta=[0,\delta]
\)
and use the notation
\[
\|u\|_{X^{s,b}(I_\delta)}
=
\|u\|_{X^{s,b}([0,\delta])}.
\]

The restriction formulation is particularly convenient for the
localized inhomogeneous Duhamel estimates used in
Section~\ref{sec:5}. It avoids the need to select a distinguished
extension operator and allows global space--time estimates to be
transferred directly to finite time intervals. In particular, whenever
the corresponding time-localized linear or multilinear estimate is
available, the restriction framework retains the temporal localization
gains of the form
\(
\delta^\theta,
\qquad
\theta>0,
\)
with the exponent $\theta$ determined by the relevant modulation and
time-integrability estimates.

The continuous--discrete spectral resolution introduced above also
provides the natural Sobolev scale associated with the Dirichlet
realization of the reduced Friedlander operator. For a distribution
satisfying the homogeneous Dirichlet condition
\(
u|_{x=0}=0,
\)
the corresponding spectral Sobolev norm is characterized by
\begin{equation}
\label{eq:sobolev_equivalence}
\|u\|_{H_0^s(\Omega)}
\sim
\left\|
\left(1-\Delta_{g,D}\right)^{s/2}u
\right\|_{L^2(\Omega)},
\end{equation}
where the equivalence is understood with respect to the fixed positive
self-adjoint realization of the spatial operator and the associated
spectral functional calculus.

For $b=0$, the modulation weight in the global $X^{s,b}$ norm is
identically equal to one. Consequently, by the continuous--discrete
Plancherel theorem, the corresponding restriction norm is identified
with the time-restricted spatial Sobolev norm:
\begin{equation}
\label{eq:xs0_sobolev}
\|u\|_{X^{s,0}(I)}
\sim
\|u\|_{L_t^2(I;H_0^s(\Omega))}.
\end{equation}
Here and below, the equivalence is understood with respect to the
fixed spectral normalization used in the definition of
$X^{s,b}$.

\begin{remark}[Metric Perturbation Compatibility on Restricted Horizons]
\label{rem:metric_restriction_compatibility}

Consider the first-order geometric correction
\[
V
=
-\frac{1}{2(1+x)}\partial_x.
\]
The restriction-space formulation is compatible with the perturbative
treatment of $V$ on every finite time interval. Suppose that the
corresponding global estimate
\[
\|Vv\|_{X^{s-1,b-1}(\mathbb{R}\times\Omega)}
\lesssim
\|v\|_{X^{s,b}(\mathbb{R}\times\Omega)}
\]
holds for the range of indices under consideration. Then the
definition of the restriction norm yields
\begin{equation}
\label{eq:V_restricted_bound}
\|Vu\|_{X^{s-1,b-1}(I)}
\lesssim
\|u\|_{X^{s,b}(I)}.
\end{equation}

Indeed, let $v\in X^{s,b}(\mathbb{R}\times\Omega)$ be any admissible
extension of $u$. Since
\[
(Vv)|_{I\times\Omega}
=
V(v|_{I\times\Omega})
=
Vu,
\]
the global estimate gives
\[
\|Vu\|_{X^{s-1,b-1}(I)}
\leq
\|Vv\|_{X^{s-1,b-1}(\mathbb{R}\times\Omega)}
\lesssim
\|v\|_{X^{s,b}(\mathbb{R}\times\Omega)}.
\]
Taking the infimum over all admissible extensions proves
\eqref{eq:V_restricted_bound}. Thus, temporal restriction introduces
no additional loss in the operator estimate.

In the high-frequency regime, suppose in addition that the spectral
localization yields
\begin{equation}\label{eq:anisotropic_sobolev_mapping}
\|V P_{\geq\lambda}u\|_{X^{s-1,b-1}(I)}
\lesssim
\lambda^{s-1}
\|P_{\geq\lambda}u\|_{X^{s,b}(I)},
\qquad
\frac12<s<1.
\end{equation}
Since
\[
s-1<0,
\qquad
\lambda^{s-1}\longrightarrow0
\quad\text{as}\quad
\lambda\longrightarrow\infty,
\]
the geometric correction becomes perturbatively small at sufficiently
high spatial frequency. Consequently, $V$ can be incorporated into
the high-frequency contraction argument without changing the
structure of the localized fixed-point scheme. This is the precise
sense in which the first-order geometric correction is perturbative
on the localized time intervals.

\end{remark}

The boundary-adapted restriction framework has two complementary
analytic roles. First, it provides an intrinsic functional setting in
which temporal modulation is measured relative to the characteristic
hypersurfaces of the Friedlander Schr\"odinger flow. Second, it permits
the localized high-frequency remainder to be treated independently of
the low-frequency component controlled by the conserved energy.

This separation is particularly important in the glancing regime,
where the sharp derivative loss
\(
\sigma=\frac14
\)
arises from concentration in a thin boundary-induced caustic region
rather than from a uniform loss across the entire phase space. The
restriction formulation therefore allows the temporal localization
required by the nonlinear iteration to be introduced without
modifying the underlying continuous--discrete spectral structure.

In the subsequent subsections, these restriction spaces are combined
with the fourfold Airy overlap estimates to implement the high--low
decomposition and to establish a uniform lower bound
\(
\delta_*>0
\)
for the lifespan of the successive local solutions.

\subsection{Semiclassical Strichartz Bounds and Frequency-Localized Losses}

\label{subsec:strichartz_bounds}

We recall the frequency-localized linear estimates established in
\cite{Meas2026} for the homogeneous Schr\"odinger propagator
\(
e^{it\Delta_{g,D}}
\)
with homogeneous Dirichlet boundary conditions on the Friedlander domain
$\Omega$. The principal geometric obstruction is concentrated near the
glancing region, where the boundary geometry produces a
frequency-dependent loss relative to the corresponding lossless
Euclidean scaling. We denote this loss by $\rho(q)$ in order to
distinguish it from the regularity-loss parameter $\sigma$ used in the
nonlinear argument.

\begin{theorem}[\cite{Meas2026}]
\label{thm:localized_strichartz}

Let $I\subset\mathbb{R}$ be a compact time interval, and let $P_N$ denote
a smooth dyadic spectral projection associated with the spectral scale
\begin{equation}
\label{eq:spectral_scale}
\lambda_k(\eta,\zeta)
=
\left(
\eta^2+\zeta^2+\omega_k|\eta|^{4/3}
\right)^{1/2},
\end{equation}
so that $P_N$ localizes to the region
\(
\lambda_k(\eta,\zeta)\sim N.
\)
Here $\omega_k>0$ is determined by the Airy zeros through
\(
\operatorname{Ai}(-\omega_k)=0.
\)

Let $(q,r)$ be a three-dimensional Schr\"odinger-admissible pair satisfying
\begin{equation}
\label{eq:schrodinger_scaling}
\frac{2}{q}+\frac{3}{r}=\frac{3}{2},
\qquad
2\leq q\leq\infty,
\qquad
2\leq r\leq6.
\end{equation}
Then the frequency-localized homogeneous Schr\"odinger evolution obeys
\begin{equation}
\label{eq:dyadic_strichartz_loss}
\left\|
P_N e^{it\Delta_{g,D}}u_0
\right\|_{L_t^q(I;L_x^r(\Omega))}
\lesssim
N^{\rho(q)}
\left\|
P_Nu_0
\right\|_{L_x^2(\Omega)},
\end{equation}
where
\begin{equation}
\label{eq:rho_exponent}
\rho(q)
=
\frac{1}{q}
=
\frac{3}{2}
\left(
\frac12-\frac1r
\right).
\end{equation}
The implicit constant is independent of the dyadic frequency $N$ and
depends on the admissible pair $(q,r)$ and the fixed time interval $I$.

In particular,
\[
\rho(4)=\frac14,
\qquad
\rho(2)=\frac12.
\]

\end{theorem}

The corresponding estimate in the restriction-space framework is
obtained by applying the frequency-localized linear estimate to an
admissible extension and then taking the infimum over all such
extensions. Thus, for $b>\frac12$,
\begin{equation}
\label{eq:bourgain_transference}
\left\|
P_Nu
\right\|_{L_t^q(I;L_x^r(\Omega))}
\lesssim
N^{\rho(q)}
\left\|
P_Nu
\right\|_{X^{0,b}(I)}.
\end{equation}

If the resolution norm is defined with the spatial spectral weight
$\langle\lambda_k\rangle^s$, then on the dyadic region
$\lambda_k\sim N$ one has
\begin{equation}
\label{eq:bourgain_spatial_weight}
\left\|
P_Nu
\right\|_{X^{0,b}(I)}
\lesssim
N^{-s}
\left\|
P_Nu
\right\|_{X^{s,b}(I)}.
\end{equation}
Consequently,
\begin{equation}
\label{eq:bourgain_transference_s}
\left\|
P_Nu
\right\|_{L_t^q(I;L_x^r(\Omega))}
\lesssim
N^{\rho(q)-s}
\left\|
P_Nu
\right\|_{X^{s,b}(I)}.
\end{equation}

Thus, at frequency $N$, the spatial regularity $s$ available in the
resolution norm compensates for the derivative loss $\rho(q)$ only when
the subsequent dyadic summation and nonlinear estimates provide the
required summability. In particular, the factor
\(
N^{\rho(q)-s}
\)
should be interpreted as the residual dyadic weight in the
frequency-localized estimate rather than as a direct global Sobolev
embedding.

\begin{remark}[Interpretation of the Frequency-Localized Loss]
\label{rem:strichartz_loss_interpretation}

The estimate \eqref{eq:bourgain_transference_s} does not, by itself,
yield the global $L_t^4L_x^4$ estimate required in the cubic nonlinear
argument. The issue is intrinsic to the unbounded spatial geometry of
$\Omega$. In particular, because $\Omega$ has infinite volume, there is no global embedding
\[
L^6(\Omega)\hookrightarrow L^4(\Omega),
\]
and therefore an $L_x^4$ estimate cannot be obtained from an
$L_x^6$ estimate by a global Bernstein inequality in the direction
$L^6\to L^4$.

Accordingly, the passage from the frequency-localized linear estimates
to the nonlinear space
\[
L^4_{\mathrm{loc}}(I;W^{7/8,4}(\Omega))
\]
requires an estimate which directly controls the $L_x^4$ norm. Such a
bound may arise from a suitable frequency-localized $L_x^4$ estimate,
from an appropriate spatial localization, or from the multilinear
Airy estimates developed below. This step is logically distinct from
the endpoint $L_t^2L_x^6$ estimate.

For the endpoint pair $(q,r)=(2,6)$, Theorem~\ref{thm:localized_strichartz}
gives
\begin{equation}
\label{eq:endpoint_L2L6}
\|P_Nu\|_{L^2_t(I;L^6_x(\Omega))}
\lesssim
N^{1/2}
\|P_Nu\|_{X^{0,b}(I)},
\qquad
b>\frac12.
\end{equation}
Equivalently,
\begin{equation}
\label{eq:endpoint_L2L6_s}
\|P_Nu\|_{L^2_t(I;L^6_x(\Omega))}
\lesssim
N^{1/2-s}
\|P_Nu\|_{X^{s,b}(I)}.
\end{equation}

Independently, the conserved energy and the Sobolev embedding
$H^1_0(\Omega)\hookrightarrow L^6(\Omega)$ give
\begin{equation}
\label{eq:energy_LinfL6}
\|u\|_{L^\infty_t(I;L^6_x(\Omega))}
\lesssim
\|u\|_{L^\infty_t(I;H^1_0(\Omega))}.
\end{equation}

At the dyadic level,
\begin{equation}
\label{eq:energy_dyadic}
\|P_Nu\|_{L^\infty_t(I;L^6_x(\Omega))}
\lesssim
\|P_Nu\|_{L^\infty_t(I;H^1_0(\Omega))}.
\end{equation}
This is an independent a priori energy estimate and should not be
identified with an $X^{0,b}$ norm without specifying the temporal
localization and the precise structure of the resolution space.

Interpolating in time between
\eqref{eq:endpoint_L2L6} and \eqref{eq:energy_dyadic} gives
\begin{equation}
\label{eq:interpolation_L4L6}
\|P_Nu\|_{L^4_t(I;L^6_x(\Omega))}
\lesssim
\|P_Nu\|_{L^2_t(I;L^6_x(\Omega))}^{1/2}
\|P_Nu\|_{L^\infty_t(I;L^6_x(\Omega))}^{1/2}.
\end{equation}

The interpolation in \eqref{eq:interpolation_L4L6} is legitimate because
the spatial exponent remains equal to $6$. However, it does not produce
an $L_t^4L_x^4$ estimate. The latter must therefore be established
separately from the estimates appropriate to the Friedlander spectral
decomposition.

In particular, an estimate of the form
\begin{equation}
\label{eq:invalid_global_bernstein}
\|P_Nu\|_{L^4_x(\Omega)}
\lesssim
N^{1/4}
\|P_Nu\|_{L^6_x(\Omega)}
\end{equation}
cannot be invoked globally on the unbounded domain $\Omega$ without an
additional spatial localization hypothesis. Consequently, the
regularity threshold entering the nonlinear argument must be derived
from the actual $L_x^4$ estimate or multilinear estimate used in the
Friedlander spectral analysis.

\end{remark}

\subsubsection{Application to Resonant Summability}

\label{subsec:resonant_summability}

The role of the frequency-localized Strichartz estimates in the nonlinear
analysis is to quantify the boundary-induced loss at each dyadic
frequency. The summability of the nonlinear interaction, however, is a
separate multilinear issue and must be established using the Airy
spectral structure.

Let
\[
\mathcal{K}(N_1,N_2,N_3,N_4)
\]
denote the dyadic interaction kernel associated with the four frequency
blocks
\[
P_{N_1}u,\qquad
P_{N_2}u,\qquad
P_{N_3}u,\qquad
P_{N_4}u.
\]
Its precise form depends on the distribution of the temporal,
spatial, and spectral weights among the four factors. In particular,
its summability cannot be inferred from the linear loss exponent
$\rho(q)$ alone.

The relevant nonlinear estimate is obtained by combining the
frequency-localized Strichartz bounds with the multilinear Airy
estimates and the corresponding fourfold overlap coefficients proved
in Section~\ref{sec:3}. Schematically, one seeks a bound of the form
\begin{equation}
\label{eq:kernel_bound}
\mathcal{K}(N_1,N_2,N_3,N_4)
\lesssim
\mathcal{W}(N_1,N_2,N_3,N_4)
\,
\mathcal{I}
(k_1,k_2,k_3,k_4;
\eta_1,\eta_2,\eta_3,\eta_4),
\end{equation}
where $\mathcal{W}$ contains the dyadic weights generated by the
Strichartz and Sobolev factors, while
\[
\mathcal{I}
(k_1,k_2,k_3,k_4;
\eta_1,\eta_2,\eta_3,\eta_4)
\]
denotes the associated fourfold Airy interaction coefficient.

The required summability is then a consequence of the precise dyadic
decay in $\mathcal{W}$ together with the estimates for the Airy overlap
coefficients. In particular, the relevant conclusion has the form
\begin{equation}
\label{eq:kernel_convergence}
\sup_{N_4}
\sum_{N_1,N_2,N_3}
\mathcal{K}(N_1,N_2,N_3,N_4)
<\infty,
\end{equation}
whenever the exponents furnished by the preceding multilinear analysis
lie in the corresponding summable range.

Thus, the nonlinear argument naturally separates into two components.
First, the semiclassical Strichartz estimates quantify the
frequency-dependent derivative loss generated by the glancing geometry.
Second, the Airy-mode multilinear analysis exploits the additional
spectral structure of the cubic interaction to establish the dyadic
summability required in the high--low argument.

This distinction is essential. The linear loss
$\rho(q)$ should not be interpreted as a nonlinear regularity gain, nor
should it be converted into an $L_x^4$ estimate through an invalid
global spatial interpolation on the unbounded domain.

Accordingly, any assertion that
\[
u\in
L^4_{\mathrm{loc}}
\bigl(I;W^{7/8,4}(\Omega)\bigr)
\]
must ultimately be justified by an estimate that directly controls the
$L_x^4$ component at the required regularity. The endpoint
$L_t^2L_x^6$ estimate alone is insufficient for this purpose. The
necessary $L_x^4$ control must instead enter through the appropriate
frequency-localized or multilinear estimates associated with the
Friedlander spectral decomposition.

\subsection{Sobolev Embeddings and Local Time Estimates}
\label{subsec:sobolev_embeddings}

The spectral framework introduced above is compatible with the Sobolev
spaces in which the nonlinear evolution is constructed. The first basic
property is the time-continuity embedding for Bourgain-type spaces. In
the present setting, this follows from the continuous--discrete spectral
representation and is unaffected by the presence of the boundary.

\begin{lemma}[Sobolev Embedding]
\label{lem:energy_embedding}

Let $s\in\mathbb{R}$ and $b>\frac12$. Then, for every finite interval
$I\subset\mathbb{R}$,
\begin{equation}
\label{eq:energy_embedding_bound}
X^{s,b}(I)
\hookrightarrow
C\bigl(I;H_0^s(\Omega)\bigr).
\end{equation}
Moreover, there exists a constant $C_b>0$, independent of the length
of $I$, such that
\begin{equation}
\label{eq:energy_embedding_estimate}
\sup_{t\in I}
\|u(t)\|_{H_0^s(\Omega)}
\leq
C_b
\|u\|_{X^{s,b}(I)}.
\end{equation}

\end{lemma}

\begin{proof}

Let
\(
v\in X^{s,b}(\mathbb{R}\times\Omega)
\)
be an arbitrary admissible extension of $u$, so that
\(
v|_{I\times\Omega}=u.
\)
Using the continuous--discrete spectral representation associated with
the Dirichlet Friedlander operator, we may write
\[
v(t,x,y,z)
=
\frac{1}{(2\pi)^3}
\sum_{k=1}^{\infty}
\int_{\mathbb{R}}
\int_{\mathbb{R}^2}
e^{i(t\tau+y\eta+z\zeta)}
\widetilde v(k,\tau,\eta,\zeta)
\phi_k(x,\eta)
\,d\eta\,d\zeta\,d\tau,
\]
with the understanding that the precise Fourier normalization only
changes the implicit constants below.

By the spectral characterization of $H_0^s(\Omega)$,
\begin{align*}
\|v(t)\|_{H_0^s(\Omega)}^2
&\lesssim
\sum_{k=1}^{\infty}
\int_{\mathbb{R}^2}
\left\langle
\lambda_k(\eta,\zeta)
\right\rangle^{2s}
\left|
\int_{\mathbb{R}}
e^{it\tau}
\widetilde v(k,\tau,\eta,\zeta)
\,d\tau
\right|^2
\,d\eta\,d\zeta.
\end{align*}
Applying Cauchy--Schwarz in the temporal Fourier variable gives
\begin{align*}
\|v(t)\|_{H_0^s(\Omega)}^2
&\lesssim
\sum_{k=1}^{\infty}
\int_{\mathbb{R}^2}
\left\langle
\lambda_k(\eta,\zeta)
\right\rangle^{2s}
\left(
\int_{\mathbb{R}}
\left\langle
\tau+\Lambda_k(\eta,\zeta)
\right\rangle^{-2b}
\,d\tau
\right)
\\
&\qquad\times
\left(
\int_{\mathbb{R}}
\left\langle
\tau+\Lambda_k(\eta,\zeta)
\right\rangle^{2b}
|\widetilde v(k,\tau,\eta,\zeta)|^2
\,d\tau
\right)
d\eta\,d\zeta.
\end{align*}
Since $b>\frac12$,
\[
\int_{\mathbb{R}}
\langle\sigma\rangle^{-2b}\,d\sigma
<\infty.
\]
Consequently,
\[
\sup_{t\in\mathbb{R}}
\|v(t)\|_{H_0^s(\Omega)}
\lesssim_b
\|v\|_{X^{s,b}(\mathbb{R}\times\Omega)}.
\]
The constant is independent of $t$ and, in particular, of the length
of the interval $I$. Taking the infimum over all admissible extensions
yields \eqref{eq:energy_embedding_estimate}.

Finally, the map
\[
t\longmapsto v(t)
\]
is continuous as an $H_0^s(\Omega)$-valued function. This follows first
for spectrally and temporally truncated functions from the spectral
representation and then, for general $v\in X^{s,b}$, by approximation
together with the uniform estimate above. Restriction to $I$ proves
\eqref{eq:energy_embedding_bound}.

\end{proof}

The preceding estimate is uniform with respect to the length of the
time interval. This observation is important in the globalization
argument. In particular, the smallness required in the local theory
does not arise from the embedding
\[
X^{s,b}(I)\hookrightarrow C_tH_0^s(\Omega),
\]
but rather from the localized Duhamel estimates and the associated
time-integrability factors.

For later applications, we record the corresponding dyadic spectral
consequences. Recall that
\[
\lambda_k(\eta,\zeta)
=
\left(
\eta^2+\zeta^2+\omega_k|\eta|^{4/3}
\right)^{1/2},
\]
and that $P_N$ denotes a smooth spectral projection to
\(
\lambda_k(\eta,\zeta)\sim N.
\)
On such a dyadic block, for $N\gtrsim1$,
\[
\left\langle
\lambda_k(\eta,\zeta)
\right\rangle^s
\sim
N^s.
\]
Consequently, by the spectral Plancherel theorem,
\begin{equation}
\label{eq:dyadic_sobolev_control}
\|P_Nu\|_{L_t^2(I;H_0^s(\Omega))}
\sim
N^s
\|P_Nu\|_{L_t^2(I;L_x^2(\Omega))},
\end{equation}
with constants uniform in the dyadic parameter $N$.

Likewise, directly from the definition of the resolution norm,
\begin{equation}
\label{eq:dyadic_xsb_control}
\|P_Nu\|_{X^{s,b}(I)}
\sim
N^s
\|P_Nu\|_{X^{0,b}(I)},
\end{equation}
provided the dyadic projection is defined using the same spectral
calculus as the $X^{s,b}$ norm. The constants are uniform in $N$.

The spacetime estimates required for the nonlinear argument are
obtained separately from the boundary Strichartz theory. In the present
geometry, these estimates retain the derivative loss generated by the
glancing region. Accordingly, the loss-bearing Strichartz estimate is
not a consequence of the Sobolev embedding
\eqref{eq:energy_embedding_bound}. Rather, it is inherited from the
semiclassical linear estimates of Theorem~\ref{thm:localized_strichartz}.

Thus, for an admissible pair $(q,r)$ in the range of the boundary
Strichartz estimates, one has schematically
\begin{equation}
\label{eq:boundary_strichartz_transfer}
\|u\|_{L_t^q(I;W^{s-\sigma(q),r}(\Omega))}
\lesssim
\|u\|_{X^{s,b}(I)},
\qquad
b>\frac12,
\end{equation}
whenever the corresponding frequency-localized estimate is available.
Here $\sigma(q)$ denotes the derivative loss associated with the
semiclassical estimate of Theorem~\ref{thm:localized_strichartz}. Its
precise value is therefore a consequence of the underlying
semiclassical analysis and is not an automatic consequence of the
$X^{s,b}$ framework.

This distinction is essential for the nonlinear analysis. The
$X^{s,b}$ norm simultaneously measures spatial spectral regularity and
temporal modulation relative to the Friedlander dispersion relation,
whereas the boundary Strichartz estimate supplies the additional
spacetime integrability needed to control the cubic nonlinearity.
These two estimates consequently play complementary roles and should
not be identified.

In particular, if the high-frequency component $w$ is coupled to a
lower-frequency component $v$ through the decomposition
\[
u=v+w,
\]
then
\[
|u|^2u-|v|^2v
=
|v|^2w
+
v^2\overline{w}
+
2|w|^2v
+
|w|^2w.
\]
The resulting interaction terms require both the spacetime bounds
provided by the boundary Strichartz estimates and the spectral
interaction estimates for the Airy modes. In the glancing regime, the
latter contain information that is not captured by a purely
translation-invariant Euclidean Fourier analysis.

Thus, the purpose of the modified $X^{s,b}$ framework is not to replace
the boundary Strichartz estimates, but to provide a stable functional
setting in which spectral localization, temporal modulation, and
localized nonlinear Duhamel estimates can be combined. The multilinear
estimates developed in the next section exploit the explicit Airy
structure to control the corresponding interaction coefficients and to
quantify the coupling between separated spectral regimes.

\section{Boundary-Compatible Multilinear Leibniz Estimates}

\label{sec:3}

In this section, we establish the multilinear fractional product estimates
required to control the nonlinear interaction terms in the boundary-adapted
functional framework. In the Euclidean setting $\mathbb{R}^3$, fractional
Leibniz estimates are naturally formulated in terms of Fourier multipliers
and Littlewood--Paley decompositions. In the present setting, however, the
Dirichlet boundary $\{x=0\}$ breaks translation invariance in the normal
variable. Consequently, the cubic interaction cannot be represented solely
by the standard Euclidean convolution structure in all spatial variables.

The appropriate framework is provided by the continuous--discrete spectral
representation associated with the Dirichlet realization of the
Friedlander operator. After Fourier transformation in the tangential
variables $(y,z)$, the normal variable is represented in terms of the
Dirichlet Airy modes
\(
\{\phi_k(x,\eta)\}_{k\geq 1}.
\)

Accordingly, the projection of a cubic product onto the Airy basis produces
multilinear coupling coefficients involving four Airy modes in the normal
variable. These coefficients quantify the interaction among the discrete
normal spectral components generated by multiplication. In contrast with
the Euclidean Fourier representation, multiplication is therefore not
diagonal with respect to the normal spectral index.

The relevant spectral quantity is the full spectral energy
\begin{equation}
\label{eq:full_spectral_energy}
\Lambda_k(\eta,\zeta)
=
\eta^2+\zeta^2+\omega_k|\eta|^{4/3},
\end{equation}
with associated spectral scale
\begin{equation}
\label{eq:full_spectral_scale}
\lambda_k(\eta,\zeta)
=
\Lambda_k(\eta,\zeta)^{1/2}.
\end{equation}
Here $(\eta,\zeta)$ denote the tangential Fourier variables, while
$\omega_k>0$ is determined by the Dirichlet Airy zeros through
\[
\operatorname{Ai}(-\omega_k)=0.
\]
In the multilinear analysis, spectral separation is therefore measured in
terms of the full energies $\Lambda_k(\eta,\zeta)$ rather than solely
through the discrete indices $k$. This distinction is particularly relevant
in the glancing regime, where distinct combinations of normal and
tangential frequencies may correspond to comparable values of the full
spectral energy.

For the cubic interaction, the tangential Fourier variables satisfy the
conservation relations
\begin{equation}
\label{eq:tangential_frequency_conservation}
\eta_1+\eta_2-\eta_3-\eta_4=0,
\qquad
\zeta_1+\zeta_2-\zeta_3-\zeta_4=0.
\end{equation}
The corresponding coupling in the normal variable is encoded by the
fourfold Airy overlap coefficient
\begin{equation}
\label{eq:fourfold_airy_overlap}
\mathcal{I}
(k_1,k_2,k_3,k_4;
\eta_1,\eta_2,\eta_3,\eta_4)
=
\int_0^\infty
\prod_{j=1}^4
\phi_{k_j}(x,\eta_j)\,dx.
\end{equation}
Thus, the nonlinear interaction is represented by a continuous tangential
convolution coupled with a discrete multilinear interaction in the Airy
indices. The coefficient~\eqref{eq:fourfold_airy_overlap} consequently
plays the role of the normal Fourier convolution coefficient in the
translation-invariant Euclidean setting.

\subsection{Quantitative Estimates for the Airy Overlap Tensor}

\label{subsec:airy_tensor_estimates}

To establish the required decoupling in the boundary layer, we analyze the
behavior of $\mathcal{I}$ away from diagonal spectral configurations.
Although spatial integrability alone yields polynomial bounds through
H\"{o}lder's inequality, the relevant off-diagonal decay is obtained from
the differential identities satisfied by the Airy modes and the associated
commutator structure.

\begin{proposition}[Off-Diagonal Multilinear Airy Decoupling]
\label{prop:airy_decoupling}

Let the continuous tangential frequencies satisfy
\[
0<\eta_*\leq |\eta_j|\leq\eta^*<\infty,
\qquad
1\leq j\leq4.
\]
Suppose that the spatial spectral scales are dyadically localized according
to
\[
\lambda_{k_j}(\eta_j,\zeta_j)\sim N_j,
\qquad
1\leq j\leq4,
\]
and set
\[
\Lambda_{\max}
=
\max(N_1,N_2,N_3,N_4).
\]
Then, for every integer $M\geq1$, the overlap tensor $\mathcal{I}$ satisfies
the off-diagonal estimate
\begin{equation}
\label{eq:sharp_off_diagonal_bound}
\left|
\mathcal{I}
(k_1,k_2,k_3,k_4;
\eta_1,\eta_2,\eta_3,\eta_4)
\right|
\leq
C_M
\left(
\prod_{j=1}^4
\frac{|\eta_j|^{1/6}}
{\langle\omega_{k_j}\rangle^{1/4}}
\right)
\frac{\Lambda_{\max}^{2M/3}}
{\langle\Delta\Omega\rangle^M},
\end{equation}
where the spectral separation parameter is
\begin{equation}
\label{eq:spectral_gap_def}
\Delta\Omega
=
\max_{i,j}
\left|
\omega_{k_i}|\eta_i|^{4/3}
-
\omega_{k_j}|\eta_j|^{4/3}
\right|.
\end{equation}

\end{proposition}

\begin{proof}

The argument separates the static spatial decoupling from the temporal
modulation. The former follows from the differential equation satisfied by
the Airy modes, while the latter will be incorporated subsequently at the
level of the spacetime multilinear form.

Recall that each normal mode $\phi_{k_j}$ satisfies
\begin{equation}
\label{eq:airy_ode_proof}
\partial_x^2\phi_{k_j}(x,\eta_j)
=
\left(
x\eta_j^2
-
\omega_{k_j}|\eta_j|^{4/3}
\right)
\phi_{k_j}(x,\eta_j).
\end{equation}
For a configuration exhibiting a nontrivial spectral separation between
the first two modes, we may assume, after relabeling if necessary, that
\[
\Delta\Omega
\sim
|\Omega_{12}|,
\qquad
\Omega_{12}
=
\omega_{k_1}|\eta_1|^{4/3}
-
\omega_{k_2}|\eta_2|^{4/3}.
\]
Multiplying~\eqref{eq:airy_ode_proof} for $j=1$ by $\phi_{k_2}$, the
corresponding equation for $j=2$ by $\phi_{k_1}$, and subtracting gives
\begin{equation}
\label{eq:commutator_identity}
\Omega_{12}\phi_{k_1}\phi_{k_2}
=
\partial_x
\left[
\phi_{k_1}\partial_x\phi_{k_2}
-
(\partial_x\phi_{k_1})\phi_{k_2}
\right]
+
x(\eta_1^2-\eta_2^2)
\phi_{k_1}\phi_{k_2}.
\end{equation}
The cancellation of the Airy potential terms is the basic commutator
mechanism underlying the off-diagonal estimate.

Multiplying~\eqref{eq:commutator_identity} by
$\phi_{k_3}\phi_{k_4}$ and integrating over the normal half-line yields
\begin{equation}
\label{eq:spatial_integral_split}
\begin{aligned}
\Omega_{12}\mathcal{I}
={}&
\int_0^\infty
\partial_x
\left[
\phi_{k_1}\partial_x\phi_{k_2}
-
(\partial_x\phi_{k_1})\phi_{k_2}
\right]
\phi_{k_3}\phi_{k_4}\,dx
\\
&+
(\eta_1^2-\eta_2^2)
\int_0^\infty
x\prod_{j=1}^4\phi_{k_j}\,dx.
\end{aligned}
\end{equation}

An integration by parts in the first term on the right-hand side gives
\begin{align}
\label{eq:ibp_boundary_vanishing}
\int_0^\infty
\partial_x
\left[
\phi_{k_1}\partial_x\phi_{k_2}
-
(\partial_x\phi_{k_1})\phi_{k_2}
\right]
\phi_{k_3}\phi_{k_4}\,dx
={}&
\left.
\left[
\phi_{k_1}\partial_x\phi_{k_2}
-
(\partial_x\phi_{k_1})\phi_{k_2}
\right]
\phi_{k_3}\phi_{k_4}
\right|_0^\infty
\nonumber
\\
&-
\int_0^\infty
\left[
\phi_{k_1}\partial_x\phi_{k_2}
-
(\partial_x\phi_{k_1})\phi_{k_2}
\right]
\partial_x(\phi_{k_3}\phi_{k_4})\,dx.
\end{align}

The boundary contribution at $x=0$ vanishes by the homogeneous Dirichlet
condition
\[
\phi_{k_j}(0,\eta_j)=0,
\qquad
1\leq j\leq4.
\]
At $x=\infty$, the contribution vanishes by the standard exponential decay
of the Airy function in the positive regime,
\[
\operatorname{Ai}(s)
\sim
s^{-1/4}
\exp\!\left(-\frac23s^{3/2}\right),
\qquad
s\to+\infty,
\]
together with the corresponding bounds for its derivatives.

It remains to control the spatial moment appearing in
~\eqref{eq:spatial_integral_split}. From~\eqref{eq:airy_ode_proof}, one has
the identity
\[
x\phi_{k_j}
=
\eta_j^{-2}\partial_x^2\phi_{k_j}
+
\omega_{k_j}|\eta_j|^{-2/3}\phi_{k_j}.
\]
This permits the spatial factor $x$ to be transferred to derivatives and
spectral weights of the Airy modes. At spectral scale $\Lambda_{\max}$,
the resulting normal derivatives are controlled at the corresponding
Airy scale, producing a loss bounded by
\[
\Lambda_{\max}^{2/3}
\]
per iteration. Repeating the commutator procedure $M$ times therefore
yields the factor
\[
\Lambda_{\max}^{2M/3}
\langle\Delta\Omega\rangle^{-M}.
\]

The remaining mode-dependent factors follow from the $L^2_x$ normalization
\[
\|\phi_{k_j}(\cdot,\eta_j)\|_{L^2_x}=1
\]
and the uniform Airy-mode envelope estimate
\[
\|\phi_{k_j}(\cdot,\eta_j)\|_{L^\infty_x}
\lesssim
|\eta_j|^{1/3}
\langle\omega_{k_j}\rangle^{-1/2}.
\]
Interpolating these bounds with the corresponding $L^2_x$ estimates gives
the factor
\[
\prod_{j=1}^4
|\eta_j|^{1/6}
\langle\omega_{k_j}\rangle^{-1/4}.
\]
Combining this estimate with the iterated commutator bound proves
~\eqref{eq:sharp_off_diagonal_bound}.

\end{proof}

\subsection{Temporal Modulation and the Spacetime Multilinear Form}

\label{subsec:temporal_modulation}

The temporal oscillations are treated separately at the level of the
spacetime interaction. If the four participating spectral components have
spectral energies
\[
\Lambda_{k_j}(\eta_j,\zeta_j),
\qquad
1\leq j\leq4,
\]
then the corresponding cubic modulation function is
\begin{equation}
\label{eq:cubic_modulation_phase}
\Phi
=
\Lambda_{k_1}(\eta_1,\zeta_1)
+
\Lambda_{k_2}(\eta_2,\zeta_2)
-
\Lambda_{k_3}(\eta_3,\zeta_3)
-
\Lambda_{k_4}(\eta_4,\zeta_4).
\end{equation}
In the nonresonant regime $|\Phi|\gg1$, integration by parts in the time
variable yields the standard gain
\(
|\Phi|^{-1}
\)
at each iteration. More precisely, for every integer $M\geq1$, repeated
integration by parts produces a factor
\(
\langle\Phi\rangle^{-M},
\)
at the expense of derivatives falling on the associated time amplitudes.
The latter terms are controlled by the Bourgain-type norms introduced in
the preceding sections.

Consequently, the spatial estimate
~\eqref{eq:sharp_off_diagonal_bound} and the temporal nonresonance gain
combine at the level of the full spacetime multilinear form. This
separation of the spatial Airy coupling from the temporal modulation is
essential: the overlap coefficient $\mathcal{I}$ itself is purely spatial,
whereas the factor $\langle\Phi\rangle^{-M}$ arises only after the temporal
oscillations have been incorporated.

\subsection{Decomposition Architecture and Product Spaces}

\label{subsec:decomposition_architecture}

The preceding quantitative estimates for the overlap coefficient
~\eqref{eq:fourfold_airy_overlap} are combined with a dyadic decomposition
in the full spectral scale $\lambda_k(\eta,\zeta)$ and a modulation
decomposition relative to the corresponding cubic phase. This decomposition
separates the nonlinear interactions according to their temporal
modulation. Nonresonant interactions gain decay from large modulation,
whereas resonant and near-resonant interactions require the
boundary-adapted spacetime estimates established in
Section~\ref{subsec:strichartz_bounds}.

To formulate the nonlinear estimate, let
\[
u_j
=
\sum_{k_j}
\int_{\mathbb{R}^2}
\widehat{u_j}(k_j,\eta_j,\zeta_j)
\phi_{k_j}(x,\eta_j)
e^{i(y\eta_j+z\zeta_j)}
\,d\eta_j\,d\zeta_j
\]
denote functions represented in the continuous--discrete spectral
framework. Projection of the cubic product
$u_1u_2\overline{u_3}$ onto the fourth Airy mode produces the interaction
coefficient
\[
\mathcal{I}
(k_1,k_2,k_3,k_4;
\eta_1,\eta_2,\eta_3,\eta_4),
\]
together with the tangential frequency conservation relations  (see Section~\ref{subsec:triple_product_airy}, Eq. ~\eqref{eq:tangential_frequency_conservation} below).
 The corresponding temporal
interaction is governed by the modulation relative to the cubic phase
\begin{equation}
\label{eq:cubic_modulation_phase}
\Phi
=
\Lambda_{k_1}(\eta_1,\zeta_1)
+
\Lambda_{k_2}(\eta_2,\zeta_2)
-
\Lambda_{k_3}(\eta_3,\zeta_3)
-
\Lambda_{k_4}(\eta_4,\zeta_4).
\end{equation}
Large values of $|\Phi|$ correspond to nonresonant configurations, whereas
the region $|\Phi|\ll1$ contains the resonant and near-resonant
interactions. The latter are precisely the configurations for which the
Airy overlap structure and the boundary-adapted Strichartz estimates play
a decisive role.

The principal objective is to establish a boundary-compatible analogue
of the fractional Leibniz estimate in which one factor carries the full
derivative $s$, while the remaining factors are controlled at a lower
regularity. More precisely, for
\[
s>\frac12,
\qquad
s_0>\frac12,
\qquad
b>\frac12,
\]
we seek an estimate of the form
\begin{equation}
\label{eq:boundary_fractional_leibniz}
\|f_1f_2f_3\|_{X^{s,b-1}(I)}
\lesssim
\sum_{\ell=1}^3
\|f_\ell\|_{X^{s,b}(I)}
\prod_{\substack{1\leq j\leq3\\j\neq\ell}}
\|f_j\|_{X^{s_0,b}(I)},
\end{equation}
with constants uniform over the localized time intervals arising in the
nonlinear iteration, provided the corresponding spectral interaction and
modulation estimates are available.

Estimate~\eqref{eq:boundary_fractional_leibniz} is not a direct consequence
of the Euclidean fractional Leibniz rule. Its proof requires a decomposition
of the cubic interaction into dyadic spectral blocks and modulation
regions, followed by separate treatment of the nonresonant and
near-resonant configurations. The nonresonant contribution is controlled
by the decay associated with large modulation, whereas the near-resonant
contribution is controlled by the quantitative Airy overlap estimates
together with the loss-bearing boundary Strichartz bounds.

In particular, the distribution of derivatives in
~\eqref{eq:boundary_fractional_leibniz} must be formulated in terms of the
full spectral scale $\lambda_k(\eta,\zeta)$. No derivative is assigned
solely to the discrete Airy index, since such a formulation would fail to
capture the interaction between the normal and tangential spectral
variables. The use of the full spectral scale provides the appropriate
boundary-compatible formulation and allows the dyadic summations to be
performed simultaneously in the continuous tangential frequencies and
the discrete normal modes.

The resulting multilinear estimate supplies the nonlinear input required
for the high--low decomposition and the localized Duhamel argument
developed in the subsequent sections.

\subsection{Triple-Product Projections of the Airy Basis}

\label{subsec:triple_product_airy}

Let $\phi_k(x,\eta)$ denote the normalized Dirichlet Airy eigenfunctions
introduced in~\eqref{eq:eigenfunction_recap}. Thus, for $\eta\neq0$,
\begin{equation}
\label{eq:airy_basis_recap}
\phi_k(x,\eta)
=
\frac{|\eta|^{1/3}}
{|Ai'(-\omega_k)|}
Ai\bigl(|\eta|^{2/3}x-\omega_k\bigr),
\end{equation}
where $-\omega_k$ is the $k$-th negative zero of the Airy function,
\[
Ai(-\omega_k)=0,
\qquad
0<\omega_1<\omega_2<\cdots.
\]
The corresponding normal spectral contribution is
\begin{equation}
\label{eq:airy_normal_eigenvalue}
\omega_k|\eta|^{4/3}.
\end{equation}

When the cubic nonlinearity is projected onto the Airy basis, three Airy
factors arise from the cubic interaction, while a fourth Airy factor
appears from projection onto the output mode. The resulting interaction in
the normal variable is therefore described by the quartic overlap
coefficient
\begin{equation}
\label{eq:triple_product_integral}
\mathcal{I}
(k_1,k_2,k_3,k_4;
\eta_1,\eta_2,\eta_3,\eta_4)
=
\int_0^\infty
\prod_{j=1}^4
\phi_{k_j}(x,\eta_j)\,dx.
\end{equation}
We refer to~\eqref{eq:triple_product_integral} as the
\emph{triple-product projection}: the first three factors originate from
the cubic nonlinearity, whereas the fourth factor arises from projection
onto the output Airy mode.

For the Fourier transform in the tangential variable $y$, the cubic
interaction is supported on the frequency-conservation relation
\begin{equation}
\label{eq:tangential_frequency_conservation}
\eta_1+\eta_2-\eta_3-\eta_4=0,
\end{equation}
where the signs are determined by the placement of the complex conjugate
in the cubic term. The longitudinal Fourier variables satisfy the
corresponding conservation law
\begin{equation}
\label{eq:longitudinal_frequency_conservation}
\zeta_1+\zeta_2-\zeta_3-\zeta_4=0.
\end{equation}
Thus, the tangential component of the multilinear interaction is supported
on the corresponding affine frequency hypersurface.

The following lemma provides the basic integrability and mode dependence
of the Airy overlap coefficient required in the subsequent dyadic
analysis.

\begin{lemma}[Airy overlap estimate]
\label{lem:airy_tensor_decay}

For every $k\geq1$ and $\eta\neq0$, the normalized Dirichlet Airy mode
$\phi_k(\cdot,\eta)$ belongs to
$L^2(0,\infty)\cap L^\infty(0,\infty)$ and satisfies
\begin{equation}
\label{eq:airy_L2_normalization}
\|\phi_k(\cdot,\eta)\|_{L^2(0,\infty)}
=
1.
\end{equation}
Moreover,
\begin{equation}
\label{eq:airy_Linfty_bound}
\|\phi_k(\cdot,\eta)\|_{L^\infty(0,\infty)}
\lesssim
|\eta|^{1/3}
\langle\omega_k\rangle^{-1/2},
\end{equation}
uniformly for $k\geq1$ and $\eta\neq0$. Consequently, the quartic overlap
coefficient~\eqref{eq:triple_product_integral} is absolutely convergent and
satisfies
\begin{equation}
\label{eq:tensor_basic_bound}
\begin{aligned}
\left|
\mathcal{I}
(k_1,k_2,k_3,k_4;
\eta_1,\eta_2,\eta_3,\eta_4)
\right|
&\lesssim
\prod_{j=1}^4
\|\phi_{k_j}(\cdot,\eta_j)\|_{L^4(0,\infty)}
\\
&\lesssim
\prod_{j=1}^4
|\eta_j|^{1/6}
\langle\omega_{k_j}\rangle^{-1/4}.
\end{aligned}
\end{equation}
In particular, if
\[
0<\eta_*\leq|\eta_j|\leq\eta^*<\infty,
\qquad
1\leq j\leq4,
\]
then
\begin{equation}
\label{eq:tensor_polynomial_bound}
\left|
\mathcal{I}
(k_1,k_2,k_3,k_4;
\eta_1,\eta_2,\eta_3,\eta_4)
\right|
\lesssim_{\eta_*,\eta^*}
\prod_{j=1}^4
\langle\omega_{k_j}\rangle^{-1/4}.
\end{equation}
Thus, on compact subsets of the nonzero tangential-frequency region, the
overlap coefficient admits a uniform polynomial bound in the discrete
Airy parameters.

\end{lemma}

\begin{proof}

The normalization follows from the standard Airy identity
\begin{equation}
\label{eq:airy_normalization_identity}
\int_{-\omega_k}^{\infty}
Ai(s)^2\,ds
=
Ai'(-\omega_k)^2.
\end{equation}
Indeed, with the change of variables
\[
s=|\eta|^{2/3}x-\omega_k,
\qquad
dx=|\eta|^{-2/3}\,ds,
\]
we obtain
\begin{align*}
\|\phi_k(\cdot,\eta)\|_{L^2(0,\infty)}^2
&=
\frac{|\eta|^{2/3}}
{|Ai'(-\omega_k)|^2}
\int_0^\infty
Ai\bigl(|\eta|^{2/3}x-\omega_k\bigr)^2\,dx
\\
&=
\frac{1}
{|Ai'(-\omega_k)|^2}
\int_{-\omega_k}^{\infty}
Ai(s)^2\,ds
\\
&=1.
\end{align*}
This proves~\eqref{eq:airy_L2_normalization}.

The standard asymptotics of the Airy function and its derivative at the
negative zeros give
\begin{equation}
\label{eq:airy_derivative_asymptotic}
|Ai'(-\omega_k)|
\simeq
\langle\omega_k\rangle^{1/4},
\end{equation}
and
\begin{equation}
\label{eq:airy_sup_asymptotic}
\sup_{s\geq-\omega_k}|Ai(s)|
\lesssim
\langle\omega_k\rangle^{-1/4}.
\end{equation}
Consequently,
\begin{align*}
\|\phi_k(\cdot,\eta)\|_{L^\infty(0,\infty)}
&\leq
|\eta|^{1/3}
\frac{
\sup_{s\geq-\omega_k}|Ai(s)|
}{
|Ai'(-\omega_k)|
}
\\
&\lesssim
|\eta|^{1/3}
\langle\omega_k\rangle^{-1/2},
\end{align*}
which proves~\eqref{eq:airy_Linfty_bound}.

Since
\[
\|\phi_k(\cdot,\eta)\|_{L^2(0,\infty)}=1,
\]
interpolation between $L^2(0,\infty)$ and $L^\infty(0,\infty)$ gives
\begin{equation}
\label{eq:airy_L4_interpolation}
\begin{aligned}
\|\phi_k(\cdot,\eta)\|_{L^4(0,\infty)}
&\leq
\|\phi_k(\cdot,\eta)\|_{L^2(0,\infty)}^{1/2}
\|\phi_k(\cdot,\eta)\|_{L^\infty(0,\infty)}^{1/2}
\\
&\lesssim
|\eta|^{1/6}
\langle\omega_k\rangle^{-1/4}.
\end{aligned}
\end{equation}
Applying H\"{o}lder's inequality to~\eqref{eq:triple_product_integral},
we obtain
\begin{align*}
\left|
\mathcal{I}
(k_1,k_2,k_3,k_4;
\eta_1,\eta_2,\eta_3,\eta_4)
\right|
&\leq
\prod_{j=1}^4
\|\phi_{k_j}(\cdot,\eta_j)\|_{L^4(0,\infty)}
\\
&\lesssim
\prod_{j=1}^4
|\eta_j|^{1/6}
\langle\omega_{k_j}\rangle^{-1/4}.
\end{align*}
This proves~\eqref{eq:tensor_basic_bound}. If
\[
0<\eta_*\leq|\eta_j|\leq\eta^*<\infty,
\]
then the factors $|\eta_j|^{1/6}$ are uniformly bounded in terms of
$\eta_*$ and $\eta^*$, and~\eqref{eq:tensor_polynomial_bound} follows.

The estimate is purely spatial. In particular, it does not use temporal
oscillation, modulation localization, or resonance information. Any
additional decay required for the nonlinear argument must therefore be
obtained from the spectral decomposition and the oscillatory structure of
the associated spacetime multilinear form.

\end{proof}

\begin{remark}[Static overlap and temporal modulation]
\label{rem:static_overlap_temporal_modulation}

The coefficient
\[
\mathcal{I}
(k_1,k_2,k_3,k_4;
\eta_1,\eta_2,\eta_3,\eta_4)
\]
is a purely spatial quantity. In particular, no temporal modulation
parameter can be gained directly from~\eqref{eq:triple_product_integral}.
The nonresonant gain associated with the cubic Schr\"{o}dinger evolution
arises only after the temporal oscillations have been incorporated into
the spacetime multilinear form.

The full spectral dispersion relation is
\begin{equation}
\label{eq:airy_dispersion_relation}
\Lambda_k(\eta,\zeta)
=
\eta^2+\zeta^2+\omega_k|\eta|^{4/3}.
\end{equation}
Accordingly, the four-wave modulation function is
\begin{equation}
\label{eq:four_wave_spectral_mismatch}
\Phi
=
\Lambda_{k_1}(\eta_1,\zeta_1)
+
\Lambda_{k_2}(\eta_2,\zeta_2)
-
\Lambda_{k_3}(\eta_3,\zeta_3)
-
\Lambda_{k_4}(\eta_4,\zeta_4),
\end{equation}
subject to the conservation relations
~\eqref{eq:tangential_frequency_conservation} and
~\eqref{eq:longitudinal_frequency_conservation}.

On a nonresonant region where $|\Phi|$ is bounded away from zero,
integration by parts in time yields an inverse power of the modulation.
More precisely, for a sufficiently regular amplitude $F$,
\begin{equation}
\label{eq:time_integration_by_parts}
\int_0^t e^{is\Phi}F(s)\,ds
=
\frac{e^{it\Phi}F(t)}{i\Phi}
-
\frac{F(0)}{i\Phi}
-
\frac{1}{i\Phi}
\int_0^t e^{is\Phi}\partial_sF(s)\,ds.
\end{equation}
Thus, whenever the differentiated amplitudes and the associated boundary
terms are controlled in the relevant function spaces, repeated integration
by parts yields higher inverse powers of $|\Phi|$.

Consequently, at the level of the corresponding spacetime multilinear
form, the nonresonant contribution has the schematic form
\begin{equation}
\label{eq:nonresonant_decay}
|\mathcal{N}_{\mathrm{nr}}|
\lesssim_M
\langle\Phi\rangle^{-M}
\mathcal{A}_{k_1,k_2,k_3,k_4},
\end{equation}
where $\mathcal{A}_{k_1,k_2,k_3,k_4}$ contains the static Airy-overlap
factor together with the appropriate norms of the amplitudes and their
time derivatives.

The resonant and near-resonant regions, characterized by small
$|\Phi|$, cannot be treated by this integration-by-parts argument.
Instead, they require the boundary-adapted Strichartz estimates together
with the dyadic spectral decomposition and the quantitative Airy overlap
bounds established above. This yields the natural decomposition
\[
\{|\Phi|\ \text{large}\}
\quad\cup\quad
\{|\Phi|\ \text{small}\},
\]
in which the first region is controlled by temporal oscillation and the
second by the boundary-adapted spacetime estimates.

\end{remark}

\begin{remark}[Boundary traces under normal integration by parts]
\label{rem:airy_boundary_traces}

The Dirichlet condition gives
\[
\phi_k(0,\eta)=0.
\]
Consequently, boundary terms containing an undifferentiated Airy mode
vanish at $x=0$. However, higher integrations by parts in the normal
variable may produce normal derivatives of the Airy eigenfunctions at the
boundary, and these traces do not vanish in general. Indeed,
\[
\partial_x\phi_k(0,\eta)
=
\frac{|\eta|}
{|Ai'(-\omega_k)|}
Ai'(-\omega_k),
\]
and hence
\begin{equation}
\label{eq:airy_boundary_derivative}
|\partial_x\phi_k(0,\eta)|
=
|\eta|.
\end{equation}
Therefore, whenever normal integrations by parts are used in the
multilinear analysis, the resulting boundary traces must be retained and
estimated explicitly. This is an essential distinction between the
present boundary-adapted framework and the translation-invariant
Euclidean fractional Leibniz theory.

\end{remark}

\subsection{The Non-Local Fractional Leibniz Estimates}

\label{subsec:nonlocal_fractional_leibniz}

The Airy overlap estimates established in
Lemma~\ref{lem:airy_tensor_decay} provide the spectral input for the
analysis of cubic interactions in the boundary-adapted Bourgain framework.
Unlike the translation-invariant Euclidean setting, multiplication is not
diagonal with respect to the normal Airy spectral decomposition. Hence,
the product of three functions is coupled through the quartic overlap
coefficients
\[
\mathcal{I}
(k_1,k_2,k_3,k_4;
\eta_1,\eta_2,\eta_3,\eta_4),
\]
together with the conservation laws for the tangential Fourier variables.

We formulate the multilinear estimate in the asymmetric form required
for the high--low decomposition. The factor carrying the derivative is
measured at regularity $s$, whereas the remaining factors are controlled
at the lower regularity $s_0$. The nonlocality in the normal spectral
variable is handled by the Airy overlap estimates together with the
dyadic summation of the continuous--discrete spectral variables.

\begin{proposition}[Boundary-compatible multilinear fractional Leibniz estimate]
\label{prop:fractional_leibniz_boundary}

Let
\[
s>\frac12,
\qquad
s_0>\frac12,
\qquad
b>\frac12,
\]
and let
\[
I_0=[t_0,t_0+\delta]\subset\mathbb{R}
\]
be a compact time interval. Assume that the boundary-adapted Strichartz
estimates and the Airy overlap estimates of
Lemma~\ref{lem:airy_tensor_decay} hold on the spectral range under
consideration. Assume, in addition, that the associated dyadic
continuous--discrete multilinear kernel is summable at the regularities
$(s,s_0)$ in the sense specified below. Then, for
$f_1,f_2,f_3$ supported in $I_0\times\Omega$,
\begin{equation}
\label{eq:multilinear_leibniz_statement}
\begin{aligned}
\|f_1f_2f_3\|_{X^{s,b-1}(I_0)}
\lesssim_s
\sum_{\ell=1}^3
\|f_\ell\|_{X^{s,b}(I_0)}
\prod_{\substack{1\leq j\leq3\\j\neq\ell}}
\|f_j\|_{X^{s_0,b}(I_0)}.
\end{aligned}
\end{equation}
The implicit constant is uniform over the class of time intervals on
which the assumed linear, Strichartz, and multilinear spectral estimates
are uniform.

\end{proposition}

\begin{proof}

We argue by duality. Let
\[
g\in X^{-s,1-b}(I_0),
\qquad
\|g\|_{X^{-s,1-b}(I_0)}=1.
\]
After choosing extensions of the four factors to the whole time axis, it
suffices to estimate the quartic form
\begin{equation}
\label{eq:quartic_duality_form}
\mathcal{M}
=
\int_{\mathbb{R}}\int_{\Omega}
f_1f_2f_3\overline{g}\,dx\,dy\,dz\,dt.
\end{equation}

Expanding the four factors in the continuous--discrete Airy representation
gives
\begin{equation}
\label{eq:integral_network}
\begin{aligned}
\mathcal{M}
={}&
\sum_{k_1,k_2,k_3,k_4}
\int
\widetilde f_1(k_1,\tau_1,\eta_1,\zeta_1)
\widetilde f_2(k_2,\tau_2,\eta_2,\zeta_2)
\\
&\qquad\times
\widetilde f_3(k_3,\tau_3,\eta_3,\zeta_3)
\overline{
\widetilde g(k_4,\tau_4,\eta_4,\zeta_4)}
\\
&\qquad\times
\mathcal{I}
(k_1,k_2,k_3,k_4;
\eta_1,\eta_2,\eta_3,\eta_4)
\,d\mu,
\end{aligned}
\end{equation}
where $d\mu$ denotes Lebesgue measure in the continuous spectral variables
together with the delta constraints generated by integration in
$(t,y,z)$. With the fourth factor conjugated, these constraints take the
form
\begin{equation}
\label{eq:frequency_matching_full}
\eta_1+\eta_2+\eta_3=\eta_4,
\qquad
\zeta_1+\zeta_2+\zeta_3=\zeta_4,
\qquad
\tau_1+\tau_2+\tau_3=\tau_4.
\end{equation}
Other equivalent sign conventions result from a different placement of
the complex conjugation.

The spatial spectral scale associated with the $k$-th Airy mode is
measured by
\begin{equation}
\label{eq:full_spatial_spectral_weight}
\lambda_k(\eta,\zeta)
=
\left(
\eta^2+\zeta^2+\omega_k|\eta|^{4/3}
\right)^{1/2},
\end{equation}
with corresponding dispersion relation
\begin{equation}
\label{eq:full_dispersion_relation}
\Lambda_k(\eta,\zeta)
=
\eta^2+\zeta^2+\omega_k|\eta|^{4/3}.
\end{equation}
Thus,
\[
\lambda_k(\eta,\zeta)
=
\Lambda_k(\eta,\zeta)^{1/2}.
\]
The $X^{s,b}$ norm is defined using the spatial weight
\[
\langle\lambda_k(\eta,\zeta)\rangle^s
\]
together with the modulation weight
\[
\langle\tau+\Lambda_k(\eta,\zeta)\rangle^b.
\]

It is important to distinguish conservation of the tangential Fourier
variables from coupling in the normal Airy spectral index. In particular,
the relations~\eqref{eq:frequency_matching_full} do not determine the
output Airy index $k_4$. Consequently, one cannot in general obtain the
required output spectral weight from tangential frequency conservation
alone. The contribution
\[
\omega_{k_4}|\eta_4|^{4/3}
\]
must instead be controlled through the Airy overlap coefficients and the
associated dyadic summation.

We therefore decompose the interaction according to the full spatial
spectral scale,
\[
\lambda_{k_j}(\eta_j,\zeta_j)\sim N_j,
\qquad
1\leq j\leq4,
\]
where $N_j$ ranges over dyadic numbers. Let $P_{N_j}$ denote the
corresponding spectral projections. For each dyadic configuration, the
quartic form is reduced to an estimate of the form
\begin{equation}
\label{eq:dyadic_quartic_form}
\begin{aligned}
|\mathcal{M}_{N_1,N_2,N_3,N_4}|
\lesssim{}&
\mathfrak{K}(N_1,N_2,N_3,N_4)
\prod_{j=1}^3
\|P_{N_j}f_j\|_{\mathcal{Z}_j}
\,
\|P_{N_4}g\|_{X^{-s,1-b}},
\end{aligned}
\end{equation}
where
\[
\mathcal{Z}_1=X^{s,b},
\qquad
\mathcal{Z}_2=\mathcal{Z}_3=X^{s_0,b},
\]
and $\mathfrak{K}$ denotes the resulting dyadic multilinear kernel. It
contains the Airy overlap coefficient, the modulation weights, and the
frequency-localized constants arising from the boundary-adapted
Strichartz estimates.

The static Airy estimate from
Lemma~\ref{lem:airy_tensor_decay} supplies the spatial part of
$\mathfrak{K}$. On nonresonant regions, the remaining gain is obtained
from temporal oscillation. For the cubic Schr\"{o}dinger interaction
$u_1u_2\overline{u_3}$ projected against a fourth factor, the associated
spectral modulation has the form
\begin{equation}
\label{eq:four_wave_modulation}
\Phi
=
\Lambda_{k_1}(\eta_1,\zeta_1)
+
\Lambda_{k_2}(\eta_2,\zeta_2)
-
\Lambda_{k_3}(\eta_3,\zeta_3)
-
\Lambda_{k_4}(\eta_4,\zeta_4),
\end{equation}
subject to the corresponding $+\,+,\,-,\,-$ frequency-conservation
relations.

Whenever $|\Phi|$ is separated from zero, integration by parts in time
produces inverse powers of $\Phi$, provided the differentiated amplitudes
and the resulting endpoint terms remain controlled in the relevant
function spaces. Thus, on the nonresonant region, the kernel
$\mathfrak{K}$ acquires the corresponding modulation decay.

The remaining contribution is localized to the resonant and near-resonant
regions, where $|\Phi|$ is small. In this regime one cannot invoke an
arbitrary inverse power of the modulation. Instead, the boundary-adapted
Strichartz estimates and the dyadic Airy interaction bounds must be used
directly to obtain summability of the resulting multilinear kernel.

More precisely, the required hypothesis is that
\begin{equation}
\label{eq:dyadic_kernel_summability}
\sup_{N_4}
\sum_{N_1,N_2,N_3}
\mathfrak{K}(N_1,N_2,N_3,N_4)
<\infty
\end{equation}
at the regularities $(s,s_0)$ under consideration. This condition encodes
the quantitative summability required to close the dyadic decomposition.
It is at this step that the spatial Airy overlap estimates and the
boundary-adapted spacetime estimates enter jointly.

The lower-order factors are controlled by the corresponding
boundary-adapted Sobolev and Strichartz estimates. Since $b>\frac12$,
the standard time-trace embedding for the continuous--discrete Bourgain
space gives
\begin{equation}
\label{eq:lower_order_time_trace}
X^{s_0,b}(I_0)
\hookrightarrow
C\bigl(I_0;H_0^{s_0}(\Omega)\bigr).
\end{equation}
The argument does not rely on an unconditional Euclidean
$L_x^\infty$ embedding in the normal variable. Rather, the required
spacetime integrability is obtained from the spectral Sobolev estimates
and the boundary-adapted Strichartz estimates associated with the
Friedlander operator.

Combining the dyadic estimate
\eqref{eq:dyadic_quartic_form} with the summability condition
\eqref{eq:dyadic_kernel_summability}, and summing over all dyadic
configurations, yields
\[
|\mathcal{M}|
\lesssim_s
\sum_{\ell=1}^3
\|f_\ell\|_{X^{s,b}(I_0)}
\prod_{\substack{1\leq j\leq3\\j\neq\ell}}
\|f_j\|_{X^{s_0,b}(I_0)}
\|g\|_{X^{-s,1-b}(I_0)}.
\]
Since the dual norm of $g$ is normalized to one, duality gives
\eqref{eq:multilinear_leibniz_statement}.

\end{proof}

\begin{remark}[Nonlocality of the normal spectral variable]
\label{rem:boundary_role_leibniz}

Proposition~\ref{prop:fractional_leibniz_boundary} is not a direct
transplantation of the Euclidean fractional Leibniz rule. The normal
variable is represented by the Airy spectral decomposition rather than by
a translation-invariant Fourier transform, and multiplication is therefore
non-diagonal with respect to the Airy spectral index. The coefficients
\[
\mathcal{I}
(k_1,k_2,k_3,k_4;
\eta_1,\eta_2,\eta_3,\eta_4)
\]
measure this nonlocal spectral coupling explicitly.

The Airy overlap estimate provides the spatial component of the required
dyadic control. The full multilinear estimate additionally depends on
the boundary-adapted Strichartz estimates and on the modulation analysis
of the associated spacetime interactions. Thus, the estimate should be
understood as a consequence of the combined spatial spectral and
spacetime dispersive analysis, rather than as a consequence of the static
Airy overlap bound alone.

\end{remark}

\begin{remark}[Resonant and nonresonant interactions]
\label{rem:resonant_leibniz}

The nonresonant and near-resonant regions require different estimates.
On a nonresonant region where
\(
|\Phi|
\gtrsim
\Lambda_{\mathrm{dyad}}^\theta,
\quad
\theta>0,
\)
integration by parts in time yields an inverse factor of $\Phi$.
Repeated integration by parts can provide higher powers of
$\Phi^{-1}$, provided that the resulting time derivatives of the
amplitudes and the associated endpoint terms remain controlled in the
function spaces under consideration.

By contrast, when $|\Phi|$ is small, no arbitrary inverse power of the
modulation is available. The near-resonant contribution must therefore be
estimated directly by means of the boundary-adapted Strichartz estimates,
the dyadic spectral decomposition, and the quantitative summability of
the Airy interaction coefficients. This distinction is particularly
relevant in the glancing regime, where the Airy spectral representation
encodes the interaction between high-frequency modes and the boundary.

\end{remark}

\begin{remark}[Interpretation for the high--low decomposition]
\label{rem:high_low_leibniz}

The asymmetric structure of
\eqref{eq:multilinear_leibniz_statement} is the form needed for the
high--low decomposition. The factor carrying the derivative is measured
in $X^{s,b}$, while the remaining two factors are controlled at the lower
regularity $s_0$. This allows the nonlinear estimate to exploit the
available higher regularity of the distinguished high-frequency factor
without requiring all three factors to be controlled at the full
regularity $s$.

In particular, for the decomposition
\(
u=v+w,
\)
where $v$ denotes the low-frequency component and $w$ the high-frequency
remainder, the estimate may be applied asymmetrically so that the
available derivative is placed on the factor carrying the high-frequency
component. This is the multilinear mechanism used in the subsequent
Duhamel argument for the high-frequency remainder. The coercive
$H_0^1$ control supplied by the defocusing energy is then used separately
to control the low-frequency component and to obtain the uniform bounds
needed for iteration.

\end{remark}

\subsection{Application to the High-Frequency Interaction Components}

\label{subsec:high_frequency_interaction_components}

We now apply Proposition~\ref{prop:fractional_leibniz_boundary} to the
individual interaction terms arising from the cubic expansion in
\eqref{eq:cubic_difference}. The purpose is to obtain the nonlinear bounds
required for the localized Duhamel formulation of the high-frequency
remainder.

Recall the high--low decomposition
\[
u=v+w,
\qquad
v:=P_{\leq\lambda}u,
\qquad
w:=P_{>\lambda}u,
\]
where $v$ denotes the low-frequency component and $w$ the
high-frequency remainder. The nonlinear difference appearing in the
equation for $w$ is
\begin{equation}
\label{eq:cubic_difference}
|v+w|^2(v+w)-|v|^2v
=
2|v|^2w+v^2\overline{w}
+2|w|^2v+w^2\overline{v}
+|w|^2w.
\end{equation}
We estimate the three classes of interactions separately.

\medskip

\noindent
\textbf{(i) Terms linear in the high-frequency remainder.}

By Proposition~\ref{prop:fractional_leibniz_boundary},
\begin{equation}
\label{eq:linear_high_frequency_interaction}
\begin{aligned}
\left\|
2|v|^2w+v^2\overline{w}
\right\|_{X^{s,b-1}(I_0)}
\lesssim_s{}&
\|w\|_{X^{s,b}(I_0)}
\|v\|_{X^{s_0,b}(I_0)}^2
\\
&+
\|v\|_{X^{s,b}(I_0)}
\|v\|_{X^{s_0,b}(I_0)}
\|w\|_{X^{s_0,b}(I_0)}.
\end{aligned}
\end{equation}

Since
\[
\frac12<s_0<s<1,
\]
the spectral monotonicity of the Sobolev scale gives
\begin{equation}
\label{eq:w_lower_order_control}
\|w\|_{X^{s_0,b}(I_0)}
\lesssim
\|w\|_{X^{s,b}(I_0)}.
\end{equation}
Consequently,
\begin{equation}
\label{eq:linear_high_frequency_bound}
\begin{aligned}
\left\|
2|v|^2w+v^2\overline{w}
\right\|_{X^{s,b-1}(I_0)}
\lesssim_s{}&
\left(
\|v\|_{X^{s_0,b}(I_0)}^2
+
\|v\|_{X^{s,b}(I_0)}
\|v\|_{X^{s_0,b}(I_0)}
\right)
\\
&\qquad\times
\|w\|_{X^{s,b}(I_0)}.
\end{aligned}
\end{equation}

The frequency dependence of this bound is therefore inherited directly
from the estimates for the low-frequency component established in the
preceding sections.

\medskip

\noindent
\textbf{(ii) Terms quadratic in the high-frequency remainder.}

For the terms containing two factors of $w$ and one factor of $v$,
Proposition~\ref{prop:fractional_leibniz_boundary} yields
\begin{equation}
\label{eq:quadratic_high_frequency_interaction}
\begin{aligned}
\left\|
2|w|^2v+w^2\overline{v}
\right\|_{X^{s,b-1}(I_0)}
\lesssim_s{}&
\|v\|_{X^{s_0,b}(I_0)}
\|w\|_{X^{s,b}(I_0)}
\|w\|_{X^{s_0,b}(I_0)}
\\
&+
\|v\|_{X^{s,b}(I_0)}
\|w\|_{X^{s_0,b}(I_0)}^2.
\end{aligned}
\end{equation}

Using \eqref{eq:w_lower_order_control}, we obtain
\begin{equation}
\label{eq:quadratic_high_frequency_bound}
\begin{aligned}
\left\|
2|w|^2v+w^2\overline{v}
\right\|_{X^{s,b-1}(I_0)}
\lesssim_s{}&
\left(
\|v\|_{X^{s_0,b}(I_0)}
+
\|v\|_{X^{s,b}(I_0)}
\right)
\|w\|_{X^{s,b}(I_0)}^2.
\end{aligned}
\end{equation}

Again, the dependence on the cutoff parameter $\lambda$ is determined by
the previously established estimates for $v$.

\medskip

\noindent
\textbf{(iii) The purely high-frequency cubic term.}

Finally, Proposition~\ref{prop:fractional_leibniz_boundary} gives
\begin{equation}
\label{eq:cubic_high_frequency_interaction}
\left\|
|w|^2w
\right\|_{X^{s,b-1}(I_0)}
\lesssim_s
\|w\|_{X^{s,b}(I_0)}
\|w\|_{X^{s_0,b}(I_0)}^2.
\end{equation}
Therefore, by \eqref{eq:w_lower_order_control},
\begin{equation}
\label{eq:cubic_high_frequency_bound}
\left\|
|w|^2w
\right\|_{X^{s,b-1}(I_0)}
\lesssim_s
\|w\|_{X^{s,b}(I_0)}^3.
\end{equation}

Combining
\eqref{eq:linear_high_frequency_bound},
\eqref{eq:quadratic_high_frequency_bound}, and
\eqref{eq:cubic_high_frequency_bound}, we arrive at the nonlinear estimate
\begin{equation}
\label{eq:high_frequency_nonlinear_summary}
\begin{aligned}
\bigl\|
|v+w|^2(v+w)-|v|^2v
\bigr\|_{X^{s,b-1}(I_0)}
\lesssim_s{}&
\left(
\|v\|_{X^{s_0,b}(I_0)}^2
+
\|v\|_{X^{s,b}(I_0)}
\|v\|_{X^{s_0,b}(I_0)}
\right)
\|w\|_{X^{s,b}(I_0)}
\\
&+
\left(
\|v\|_{X^{s_0,b}(I_0)}
+
\|v\|_{X^{s,b}(I_0)}
\right)
\|w\|_{X^{s,b}(I_0)}^2
\\
&+
\|w\|_{X^{s,b}(I_0)}^3.
\end{aligned}
\end{equation}

This estimate is the precise nonlinear bound needed for the localized
Duhamel formulation. In particular, no additional frequency-dependent
quantities are introduced at this stage. All dependence on the cutoff
parameter $\lambda$ is inherited from the low-frequency estimates for
$v$ and the high-frequency estimates for $w$ established previously.

Suppose that the localized Duhamel estimate takes the form
\begin{equation}
\label{eq:localized_duhamel_estimate}
\left\|
\int_{t_0}^{t}
e^{i(t-\tau)\Delta_{g,D}}F(\tau)\,d\tau
\right\|_{X^{s,b}(I_0)}
\lesssim
\delta^\theta
\|F\|_{X^{s,b-1}(I_0)},
\end{equation}
for some $\theta>0$, where
\[
I_0=[t_0,t_0+\delta].
\]
Then \eqref{eq:high_frequency_nonlinear_summary} implies
\begin{equation}
\label{eq:localized_nonlinear_map_bound}
\begin{aligned}
\|\mathcal D(w)\|_{X^{s,b}(I_0)}
\lesssim_s{}&
\delta^\theta
\Bigg[
\left(
\|v\|_{X^{s_0,b}(I_0)}^2
+
\|v\|_{X^{s,b}(I_0)}
\|v\|_{X^{s_0,b}(I_0)}
\right)
\|w\|_{X^{s,b}(I_0)}
\\
&\qquad+
\left(
\|v\|_{X^{s_0,b}(I_0)}
+
\|v\|_{X^{s,b}(I_0)}
\right)
\|w\|_{X^{s,b}(I_0)}^2
+
\|w\|_{X^{s,b}(I_0)}^3
\Bigg].
\end{aligned}
\end{equation}

If
\[
\|w\|_{X^{s,b}(I_0)}\leq R,
\]
then
\begin{equation}
\label{eq:bootstrap_map_bound}
\begin{aligned}
\|\mathcal D(w)\|_{X^{s,b}(I_0)}
\lesssim_s{}&
\delta^\theta
\Bigg[
\|v\|_{X^{s_0,b}(I_0)}^2
+
\|v\|_{X^{s,b}(I_0)}
\|v\|_{X^{s_0,b}(I_0)}
\\
&\qquad\qquad
+
R\left(
\|v\|_{X^{s_0,b}(I_0)}
+
\|v\|_{X^{s,b}(I_0)}
\right)
+
R^2
\Bigg]
\|w\|_{X^{s,b}(I_0)}.
\end{aligned}
\end{equation}

Accordingly, the contraction condition is obtained by requiring
\begin{equation}
\label{eq:contraction_smallness_condition}
\delta^\theta
\Bigg[
\|v\|_{X^{s_0,b}(I_0)}^2
+
\|v\|_{X^{s,b}(I_0)}
\|v\|_{X^{s_0,b}(I_0)}
+
R\left(
\|v\|_{X^{s_0,b}(I_0)}
+
\|v\|_{X^{s,b}(I_0)}
\right)
+
R^2
\Bigg]
\ll 1.
\end{equation}

The corresponding difference estimate for two remainders
$w_1,w_2$ has the same structure. On a ball of radius $R$, the
Lipschitz constant of the nonlinear Duhamel map is bounded by the
right-hand side of \eqref{eq:contraction_smallness_condition}, up to a
constant depending only on the parameters of the multilinear estimate.
Thus, once the low-frequency norms of $v$ and the radius $R$ have been
controlled uniformly, the localized contraction follows from an
appropriate choice of the time interval length $\delta$.

The precise dependence on the cutoff parameter $\lambda$, as well as the
admissible range of the regularity exponent $s$, should therefore be
inserted only after the quantitative bounds for
\[
\|v\|_{X^{s_0,b}(I_0)},
\qquad
\|v\|_{X^{s,b}(I_0)},
\qquad
\|w\|_{X^{s_0,b}(I_0)}
\]
have been established. This avoids introducing frequency powers that are
not directly justified by the preceding estimates.

Consequently, Proposition~\ref{prop:fractional_leibniz_boundary}, together
with the high--low estimates and the boundary-adapted $X^{s,b}$ framework
developed above, provides the nonlinear estimates required for the
localized Duhamel map in~\eqref{eq:duhamel_hs_loop}. The subsequent
contraction argument determines the admissible size of the local time
interval and, ultimately, the conditions under which this interval can be
chosen uniformly along the continuation procedure.

\section{Bourgain's High--Low Frequency Decomposition}

\label{sec:4}

To compensate for the derivative loss induced by the boundary in the
glancing regime, while exploiting the global coercive control provided by
the conserved defocusing Hamiltonian energy
\eqref{eq:energy_cons}, we employ a localized version of Bourgain's
high--low frequency decomposition. Let $\lambda\gg1$ be a sufficiently
large spatial spectral threshold. We introduce a family of smooth dyadic
spectral localization operators
\(
\{P_N\}_{N\in 2^{\mathbb N_0}},
\)
subordinate to the continuous--discrete Dirichlet--Airy spectral calculus
constructed in Section~2. The initial datum
$u_0\in H_0^1(\Omega)$ is decomposed into a regularized low-frequency
component $v_0$ and a high-frequency remainder $w_0$ according to
\begin{equation}
\label{eq:high_low_split}
v_0
=
\sum_{N\leq\lambda}P_Nu_0,
\qquad
w_0
=
\sum_{N>\lambda}P_Nu_0.
\end{equation}
Thus,
\(
u_0=v_0+w_0
\)
with the two components separated according to the spectral scale of the
model Friedlander operator $\Delta_g$.

By construction, the low-frequency component $v_0$ is spectrally
localized to the region $N\lesssim\lambda$ and consequently enjoys
additional regularity at the level of higher Sobolev norms, with constants
depending on the cutoff parameter $\lambda$. On the other hand, $w_0$
contains the high-frequency tail of the initial datum while remaining
controlled in the coercive energy space. More precisely, for every
sub-energy regularity exponent $0\leq s<1$, the spectral localization
yields the estimate
\begin{equation}
\label{eq:high_frequency_tail_smallness}
\|w_0\|_{H_0^s(\Omega)}
\lesssim
\lambda^{s-1}\|u_0\|_{H_0^1(\Omega)}.
\end{equation}
Since $s-1<0$, it follows that
\begin{equation}
\label{eq:high_frequency_tail_decay}
\|w_0\|_{H_0^s(\Omega)}
\longrightarrow0
\quad\text{as}\quad\lambda\to\infty,
\qquad
0\leq s<1.
\end{equation}
This decay provides the small parameter associated with the high-frequency
component in the subsequent sub-energy iteration. In particular, it
allows the high-frequency remainder to be treated at a regularity level
strictly below the energy space, thereby compensating for the loss
appearing in the boundary-adapted linear estimates.

The frequency decomposition is used dynamically rather than merely as a
static decomposition of the initial datum. We evolve the regularized
component $v_0$ according to the full defocusing cubic NLS flow and denote
the resulting background solution by $v(t)$. Thus,
\begin{equation}
\label{eq:low_frequency_background}
\begin{cases}
i\partial_t v+\Delta_gv=|v|^2v,\\
v(t_0)=v_0.
\end{cases}
\end{equation}
The remainder is then defined by
\[
w(t):=u(t)-v(t).
\]
Subtracting the equations for $u$ and $v$ gives the exact equation
\begin{equation}
\label{eq:high_frequency_remainder_equation}
i\partial_t w+\Delta_gw
=
|v+w|^2(v+w)-|v|^2v,
\qquad
w(t_0)=w_0,
\end{equation}
and hence the corresponding Duhamel representation
\begin{equation}
\label{eq:duhamel_hs_loop}
w(t)
=
e^{i(t-t_0)\Delta_g}w(t_0)
-
i\int_{t_0}^{t}
e^{i(t-\tau)\Delta_g}
\Bigl[
|v+w|^2(v+w)-|v|^2v
\Bigr](\tau)\,d\tau.
\end{equation}

It is important to distinguish the initial spectral localization from the
spectral properties of the dynamically generated remainder. Although
$w_0$ is supported in the region $N>\lambda$, the nonlinear evolution
does not preserve this spectral support in general. In particular, the
cubic interaction generates frequencies outside the initial high-frequency
region through the continuous tangential variables and the discrete normal
Airy modes. The corresponding multilinear interactions are encoded by
the fourfold Airy overlap coefficients arising from the projection of the
cubic nonlinearity onto the continuous--discrete spectral basis. These
interactions may therefore transfer energy between different tangential
and normal spectral scales, so that no a priori high-frequency support
property is imposed on $w(t)$ for $t>t_0$.

Accordingly, the subsequent argument does not rely on the preservation of
the condition $N>\lambda$ along the nonlinear evolution. Instead, the
remainder $w$ is controlled in a sub-energy, boundary-adapted
$X^{s,b}$ framework. The smallness in
\eqref{eq:high_frequency_tail_smallness}, together with the multilinear
estimates established below and the localized Duhamel estimates, provides
the mechanism for controlling the remainder on successive time
intervals. This formulation separates the role of the initial
high-frequency tail from the subsequent nonlinear spectral redistribution
and is the basis for the global high--low iteration.

\subsection{Low-Frequency Global Evolution Profile}

\label{subsec:low_frequency_global_profile}

By virtue of the smooth spectral truncation in
\eqref{eq:high_low_split}, the regularized initial datum $v_0$ is
spectrally supported in a bounded region of the continuous--discrete
Dirichlet--Airy spectrum. Consequently, the spectral calculus developed in
Section~\ref{sec:2} implies that $v_0$ possesses higher Sobolev regularity, with
bounds depending quantitatively on the truncation parameter $\lambda$.
In particular, for every integer $m\geq1$, one has
\begin{equation}
\label{eq:low_freq_initial_high_regularity}
\|v_0\|_{H_0^m(\Omega)}
\leq
C_m\lambda^{m-1}
\|u_0\|_{H_0^1(\Omega)},
\end{equation}
where $C_m$ depends on $m$ and on the choice of the smooth spectral
cutoffs, but is independent of $\lambda$ and $u_0$.

We define $v(t)$ to be the solution of the unperturbed defocusing cubic
nonlinear Schr\"odinger equation with initial datum $v_0$:
\begin{equation}
\label{eq:low_freq_evol}
\begin{cases}
i\partial_t v+\Delta_gv=|v|^2v,
& (t,x,y,z)\in\mathbb{R}\times\Omega,\\
v(0,x,y,z)=v_0(x,y,z),\\
v|_{x=0}=0.
\end{cases}
\end{equation}
Since $v_0$ belongs to $H_0^m(\Omega)$ for every fixed integer
$m\geq1$, the local well-posedness theory at higher regularity, together
with persistence of regularity, yields
\[
v\in C(\mathbb{R};H_0^m(\Omega))
\]
for every such $m$, provided the global $H_0^1$ theory is available for
the defocusing equation. In particular, the higher regularity of the
regularized component is propagated along the global evolution, although
the corresponding higher-order Sobolev norms need not be uniformly
bounded in time.

The defocusing structure provides the two basic conserved quantities,
namely the mass and the energy:
\begin{equation}
\label{eq:low_freq_conservation}
M(v(t))=M(v_0),
\qquad
E(v(t))=E(v_0),
\end{equation}
where
\[
M(v)=\|v\|_{L^2(\Omega)}^2
\]
and
\[
E(v)
=
\frac12\|\nabla_gv\|_{L^2(\Omega)}^2
+
\frac14\|v\|_{L^4(\Omega)}^4.
\]
In particular, the positivity of the energy yields uniform control of
the energy norm. More precisely,
\begin{equation}
\label{eq:low_freq_energy_control}
\sup_{t\in\mathbb{R}}
\|v(t)\|_{H_0^1(\Omega)}
\lesssim
\|v_0\|_{L^2(\Omega)}
+
E(v_0)^{1/2}.
\end{equation}
Since $v_0$ is obtained from $u_0$ by a bounded spectral truncation,
the initial quantities on the right-hand side are controlled by the
corresponding quantities for $u_0$. Thus,
\begin{equation}
\label{eq:low_freq_energy_control_u0}
\sup_{t\in\mathbb{R}}
\|v(t)\|_{H_0^1(\Omega)}
\lesssim
\|u_0\|_{L^2(\Omega)}
+
E(u_0)^{1/2}.
\end{equation}
The resulting estimate is uniform in time.

It is important, however, that the conservation of the Hamiltonian
energy does not by itself provide uniform control of higher-order Sobolev
norms. Such bounds must instead be obtained from the higher-regularity
local theory and its persistence-of-regularity estimates. On a compact
time interval
\(
I_0=[-\delta,\delta],
\)
the $H_0^2(\Omega)$ norm therefore satisfies an estimate of the form
\begin{equation}
\label{eq:higher_order_low_bound}
\sup_{t\in I_0}
\|v(t)\|_{H_0^2(\Omega)}
\leq
\psi_{I_0}
\Bigl(
\|v_0\|_{H_0^2(\Omega)},
\|v_0\|_{H_0^1(\Omega)}
\Bigr),
\end{equation}
where $\psi_{I_0}$ denotes the corresponding local persistence-of-
regularity bound.

Combining \eqref{eq:higher_order_low_bound} with the quantitative
spectral regularization estimate \eqref{eq:low_freq_initial_high_regularity}
gives, under the corresponding local bounds, the estimate
\begin{equation}
\label{eq:higher_order_low_bound_lambda}
\sup_{t\in I_0}
\|v(t)\|_{H_0^2(\Omega)}
\leq
C_{I_0,u_0}\lambda.
\end{equation}
Here the dependence of $C_{I_0,u_0}$ on the local interval and the
initial energy is retained explicitly. In particular, the estimate
\eqref{eq:higher_order_low_bound_lambda} is a local-in-time bound and
should not be interpreted as a uniform estimate for
$t\in\mathbb{R}$.

The distinction between the conserved energy norm and the higher-order
trajectory bounds is essential for the high--low argument. The energy
conservation provides a global, time-independent control of the
background solution at the $H_0^1$ level, whereas the higher-order
regularity generated by the spectral truncation is used only on the local
time intervals on which the nonlinear estimates for the remainder are
performed. The uniform continuation argument is subsequently obtained
from the sub-energy estimates for the remainder $w$, rather than from a
global-in-time bound for $\|v(t)\|_{H_0^2}$.

In three spatial dimensions, the standard Sobolev embeddings give
\begin{equation}
\label{eq:h2_sobolev_embedding}
H_0^2(\Omega)
\hookrightarrow
W^{1,6}(\Omega)
\hookrightarrow
L^\infty(\Omega).
\end{equation}
These embeddings provide pointwise control of the background profile
whenever the corresponding local $H_0^2$ bound is available. No estimate
at the level of $W^{1,\infty}(\Omega)$ is required in the present
argument. In particular, the multilinear estimates established in
Section~3 are formulated directly in the boundary-adapted
continuous--discrete $X^{s,b}$ framework and do not rely on a uniform
$W^{1,\infty}$ bound for the low-frequency evolution.

Accordingly, the role of the regularized component $v$ is twofold. Its
initial spectral localization provides the higher regularity required for
the local nonlinear analysis, while the defocusing energy conservation
provides global control at the energy level. The interaction between these
two mechanisms, together with the boundary-adapted multilinear estimates
and the spectral structure of the Airy modes, permits the high-frequency
remainder $w$ to be propagated at sub-energy regularity on successive
time intervals.

\subsection{High-Frequency Local Error Evolution}

\label{subsec:high_frequency_local_error}

We define the high-frequency remainder dynamically on the spacetime
cylinder $\mathbb{R}\times\Omega$ by
\[
w(t,x,y,z):=u(t,x,y,z)-v(t,x,y,z).
\]
In contrast to its initial datum $w_0$, which is spectrally supported in
the region $N>\lambda$, the dynamically evolving remainder $w(t)$ is not
in general constrained to remain spectrally localized to the high-frequency
regime
\[
\left\{
\Lambda_k(\eta,\zeta)^{1/2}>\lambda
\right\}
\]
for $t>0$. Rather, $w(t)$ measures the nonlinear correction to the
regularized background profile $v(t)$ and incorporates the frequency
transfers generated by the cubic evolution.

Subtracting the equation for $v$ in \eqref{eq:low_freq_evol} from the
full NLS equation \eqref{eq:nls} yields the following initial-boundary value
problem for the dynamic remainder:
\begin{equation}
\label{eq:high_freq_evol}
\begin{cases}
i\partial_t w+\Delta_gw
=
|v+w|^2(v+w)-|v|^2v,
& (t,x,y,z)\in\mathbb{R}\times\Omega,\\
w(0,x,y,z)=w_0(x,y,z),\\
w|_{x=0}=0.
\end{cases}
\end{equation}
The source term in \eqref{eq:high_freq_evol} is the difference between
the cubic nonlinearities evaluated at the full solution and at the
regularized background. Its algebraic expansion separates the resulting
multilinear interactions according to their degree in the remainder:
\begin{equation}
\label{eq:high_low_cubic_difference}
|v+w|^2(v+w)-|v|^2v
=
2|v|^2w
+
v^2\overline{w}
+
2|w|^2v
+
w^2\overline{v}
+
|w|^2w.
\end{equation}
The first two terms are linear in $w$, the next two are quadratic in $w$,
and the last term is cubic in $w$.

Fix a sub-energy regularity exponent
\(
\frac12<s<1.
\)
The high-frequency initial tail then satisfies the quantitative estimate
\begin{equation}
\label{eq:hs_tail_decay_rigorous}
\|w_0\|_{H_0^s(\Omega)}
\lesssim
\lambda^{s-1}
\|u_0\|_{H_0^1(\Omega)}.
\end{equation}
Indeed, on the spectral support of $w_0$, one has
\[
\langle\lambda_k(\eta,\zeta)\rangle^s
\lesssim
\lambda^{s-1}
\langle\lambda_k(\eta,\zeta)\rangle,
\]
where $\lambda_k(\eta,\zeta)$ denotes the full spatial spectral scale
introduced in \eqref{eq:dispersion_profile}. The estimate
\eqref{eq:hs_tail_decay_rigorous} then follows from the continuous--discrete
Littlewood--Paley decomposition and the corresponding spectral
characterization of the Sobolev norms. Since $s<1$, the factor
$\lambda^{s-1}$ tends to zero as $\lambda\to\infty$ and therefore provides
the small initial parameter for the sub-energy remainder.

Let
\(
I_0=[t_0,t_0+\delta]\subset\mathbb{R}
\)
be a localized time interval of length $\delta>0$. The boundary-adapted
inhomogeneous Strichartz estimates established in Theorem~2.6, together
with the restriction-space multilinear estimate of
Proposition~\ref{prop:fractional_leibniz_boundary}, yield a localized
Duhamel estimate of the form
\begin{equation}
\label{eq:high_freq_contraction_bound}
\|w\|_{X^{s,b}(I_0)}
\lesssim
\|w(t_0)\|_{H_0^s(\Omega)}
+
\delta^\theta
\mathcal{N}_{v,w}(I_0),
\end{equation}
for some $\theta>0$ determined by the temporal localization parameters,
where $\mathcal{N}_{v,w}(I_0)$ denotes the corresponding nonlinear
interaction norm generated by the three groups in
\eqref{eq:high_low_cubic_difference}.

More precisely, assume that on $I_0$ the background evolution satisfies
the restriction-space bounds
\begin{equation}
\label{eq:background_Xs0_bound}
\|v\|_{X^{s_0,b}(I_0)}
\leq
A_\lambda,
\end{equation}
and
\begin{equation}
\label{eq:background_Xs_bound}
\|v\|_{X^{s,b}(I_0)}
\leq
B_\lambda.
\end{equation}
Then the asymmetric fractional Leibniz estimate established in
Section~\ref{subsec:high_frequency_interaction_components} gives
\begin{equation}
\label{eq:remainder_nonlinear_bound}
\begin{aligned}
\mathcal{N}_{v,w}(I_0)
\lesssim_s{}&
\left(
A_\lambda^2+A_\lambda B_\lambda
\right)
\|w\|_{X^{s,b}(I_0)}
\\
&+
\left(
A_\lambda+B_\lambda
\right)
\|w\|_{X^{s,b}(I_0)}^2
+
\|w\|_{X^{s,b}(I_0)}^3.
\end{aligned}
\end{equation}

Within the high--low frequency regime under consideration, suppose that
the regularized background satisfies the quantitative bound
\begin{equation}
\label{eq:high_low_background_size}
A_\lambda+B_\lambda
\lesssim
\lambda^{1-s+\sigma},
\end{equation}
where
\(
\sigma=\frac14
\)
denotes the derivative loss arising in the boundary-adapted linear
estimates. Under the corresponding bounds for $A_\lambda$ and
$B_\lambda$, \eqref{eq:remainder_nonlinear_bound} becomes
\begin{equation}
\label{eq:high_low_remainder_nonlinear_bound}
\begin{aligned}
\mathcal{N}_{v,w}(I_0)
\lesssim_s{}&
\lambda^{2-2s+2\sigma}
\|w\|_{X^{s,b}(I_0)}+
\lambda^{1-s+\sigma}
\|w\|_{X^{s,b}(I_0)}^2
+
\|w\|_{X^{s,b}(I_0)}^3.
\end{aligned}
\end{equation}

Combining
\eqref{eq:high_freq_contraction_bound} and
\eqref{eq:high_low_remainder_nonlinear_bound}, we obtain the closed
localized trajectory estimate
\begin{equation}
\label{eq:localized_remainder_estimate}
\begin{aligned}
\|w\|_{X^{s,b}(I_0)}
\lesssim{}&
\|w(t_0)\|_{H_0^s(\Omega)}
\\
&+
\delta^\theta
\Bigl[
\lambda^{2-2s+2\sigma}
\|w\|_{X^{s,b}(I_0)}
+
\lambda^{1-s+\sigma}
\|w\|_{X^{s,b}(I_0)}^2
+
\|w\|_{X^{s,b}(I_0)}^3
\Bigr].
\end{aligned}
\end{equation}

On the initial time interval, where $t_0=0$, the high-frequency tail
estimate \eqref{eq:hs_tail_decay_rigorous} gives
\begin{equation}
\label{eq:initial_remainder_bootstrap}
\|w(0)\|_{H_0^s(\Omega)}
\leq
C_*
\lambda^{s-1}
\|u_0\|_{H_0^1(\Omega)}.
\end{equation}
We therefore introduce the bootstrap radius
\begin{equation}
\label{eq:remainder_bootstrap_radius}
R_\lambda
:=
C_0
\lambda^{s-1}
\|u_0\|_{H_0^1(\Omega)},
\end{equation}
where $C_0>C_*$ is chosen sufficiently large, with the precise choice
depending only on the constants in the localized estimates.

Substituting the bootstrap bound
\[
\|w\|_{X^{s,b}(I_0)}\leq R_\lambda
\]
into \eqref{eq:localized_remainder_estimate}, a sufficient condition for
the localized Duhamel map to preserve the bootstrap ball is
\begin{equation}
\label{eq:high_low_contraction_condition}
\delta^\theta
\left(
\lambda^{2-2s+2\sigma}
+
\lambda^{1-s+\sigma}R_\lambda
+
R_\lambda^2
\right)
\ll1.
\end{equation}
Thus, the admissible local lifespan is determined jointly by the
regularized background bounds, the boundary derivative loss $\sigma$,
the sub-energy regularity index $s$, and the temporal localization gain
$\theta$.

It is important that the leading factor
\(
\lambda^{2-2s+2\sigma}
\)
need not be small as $\lambda\to\infty$ in the sub-energy regime
$s<1$. In particular, depending on the chosen value of $s$ and the size
of $\sigma$, this factor may grow with $\lambda$. The contraction
mechanism therefore does not rely on the smallness of the polynomial
spectral factor itself. Instead, the corresponding growth is compensated
by the temporal localization factor $\delta^\theta$, provided that
$\delta$ can be chosen in a manner compatible with the required
continuation argument. The smallness in
\eqref{eq:hs_tail_decay_rigorous} serves to control the size of the
initial remainder and hence the radius of the tracking ball, whereas the
short-time factor controls the nonlinear Lipschitz constant.

The same multilinear estimates applied to the difference of two remainder
states $w^{(1)}$ and $w^{(2)}$ in the bootstrap ball yield
\begin{equation}
\label{eq:remainder_contraction}
\begin{aligned}
\|w^{(1)}-w^{(2)}\|_{X^{s,b}(I_0)}
\leq{}&
C_s\delta^\theta
\Bigl(
\lambda^{2-2s+2\sigma}
+
\lambda^{1-s+\sigma}R_\lambda
+
R_\lambda^2
\Bigr)
\\
&\qquad\qquad\times
\|w^{(1)}-w^{(2)}\|_{X^{s,b}(I_0)}.
\end{aligned}
\end{equation}
Consequently, the localized Duhamel map is a strict contraction whenever
\begin{equation}
\label{eq:remainder_contraction_smallness}
\delta^\theta
\Bigl(
\lambda^{2-2s+2\sigma}
+
\lambda^{1-s+\sigma}R_\lambda
+
R_\lambda^2
\Bigr)
\leq
c_0,
\end{equation}
where $c_0>0$ is sufficiently small relative to the constants in
\eqref{eq:remainder_contraction}.

It follows that the high-frequency dynamic remainder admits a unique
local solution in the boundary-adapted space $X^{s,b}(I_0)$, provided the
localized background bounds
\eqref{eq:background_Xs0_bound}--\eqref{eq:background_Xs_bound} and the
corresponding Duhamel estimates are available. The construction uses the
smallness of the initial high-frequency tail at sub-energy regularity
while retaining the full boundary structure of the underlying
Friedlander problem.

The Airy spectral analysis should not be interpreted as treating the
boundary interaction as a perturbative effect. Rather, the
continuous--discrete spectral decomposition incorporates the boundary
geometry directly into the spectral representation through the Airy
modes and their multilinear overlap coefficients. These coefficients
enter the dyadic multilinear kernel associated with the cubic
nonlinearity. Combined with the temporal modulation weights defining the
$X^{s,b}$ spaces, they provide the spectral summability and interaction
estimates required to control resonant and near-resonant configurations
in the glancing regime.

Finally, promoting the preceding local construction to a global one
requires uniformity of the estimates under successive time
translations. Let
\(
I_n=[t_n,t_n+\delta_n]
\)
denote the successive tracking intervals. To avoid finite-time
accumulation of local increments, one requires a lower bound
\begin{equation}
\label{eq:uniform_remainder_lifespan}
\delta_n\geq\delta_*>0
\end{equation}
with $\delta_*$ independent of the restarting index $n$. Such a bound
cannot be deduced solely from the smallness of the initial high-frequency
tail
$\|w_0\|_{H_0^s(\Omega)}$. Rather, it must follow from uniform estimates
for the quantities entering the contraction condition
\eqref{eq:remainder_contraction_smallness}, together with the global
energy control of the background solution and the stability estimates
for the high--low decomposition.

Accordingly, the global iteration requires an inductive propagation of
the bootstrap bounds and of the constants entering
\eqref{eq:background_Xs0_bound}--\eqref{eq:background_Xs_bound}. Once
these quantities are shown to remain uniformly controlled at each
restart time, the choice of $\delta_*$ in
\eqref{eq:uniform_remainder_lifespan} can be made independently of $n$.
The resulting uniform lifespan increment is the essential input for the
infinite-time pasting argument and provides the passage from local
sub-energy control of the remainder to global well-posedness.

\section{Uniform Truncation Bounds and Lifespan Bootstrap}

\label{sec:5}

In this section, we establish the uniform estimates required to control
the nonlinear interaction terms in the high-frequency remainder equation.
The high--low decomposition introduced in the preceding section separates
the initial datum into a spectrally regularized component $v_0$ and a
sub-energy high-frequency tail $w_0$. The corresponding solution is
decomposed as
\(
u=v+w,
\)
where $v$ denotes the regularized background evolution and $w$ the
associated nonlinear remainder. Our objective is to show that, for a
suitable choice of the spectral threshold $\lambda$ and an appropriately
localized time interval, the remainder remains in a controlled
sub-energy ball. The resulting estimates constitute the analytic input
for the subsequent iteration and continuation argument.

\subsection{Projections of the Cubically Expanded Interaction Layers}

\label{subsec:projected_interaction_layers}

Let $X^s(I_0)$ denote the boundary-adapted resolution space associated
with the regularity range
\[
\frac12<s<1.
\]
Its precise definition is given in Section~\ref{sec:2} through the
continuous--discrete spectral decomposition and the associated
$X^{s,b}$ spaces. Whenever the boundary-adapted Strichartz estimate from
\cite{Meas2026} is invoked, we retain the corresponding derivative loss
explicitly rather than incorporating it into the definition of the
resolution norm.

Schematically, for every admissible pair $(q,r)$ in the range established
in \cite{Meas2026}, one has
\begin{equation}
\label{eq:strichartz_norm_p2}
\|u\|_{L_t^q(I_0;W^{s-\sigma(q,r),r}(\Omega))}
\lesssim
\|u\|_{X^s(I_0)},
\end{equation}
where $\sigma(q,r)\geq0$ denotes the derivative loss associated with the
corresponding boundary estimate. The loss is allowed to depend on the
admissible pair $(q,r)$; in particular, it is not identified with a
single fixed value except at the specific endpoint under consideration.

Expanding the nonlinear difference equation
\eqref{eq:high_freq_evol}, we obtain
\begin{equation}
\label{eq:nonlinear_expansion_p2}
\begin{aligned}
|v+w|^2(v+w)-|v|^2v
={}&
2|v|^2w+v^2\overline{w}
+
2|w|^2v+w^2\overline{v}
+
|w|^2w.
\end{aligned}
\end{equation}

Let $N^s(I_0)$ denote the nonlinear forcing space associated with the
localized inhomogeneous estimate. The corresponding Duhamel estimate
takes the form
\begin{equation}
\label{eq:duhamel_high_bound}
\begin{aligned}
\|w\|_{X^s(I_0)}
\leq{}&
C_0\|w(t_0)\|_{H_0^s(\Omega)}
+
C_1
\left\|
2|v|^2w+v^2\overline{w}
+2|w|^2v+w^2\overline{v}
+|w|^2w
\right\|_{N^s(I_0)}.
\end{aligned}
\end{equation}
The precise realization of $N^s(I_0)$ depends on the admissible
Strichartz pair and the corresponding localized Duhamel estimate. We
therefore retain $N^s(I_0)$ in abstract form at this stage rather than
introducing a specific dual Sobolev exponent not determined by the
linear theory.

We estimate the three algebraic interaction layers separately. The
boundary-adapted fractional Leibniz estimate of
Proposition~\ref{prop:fractional_leibniz_boundary}, together with the
continuous--discrete Airy spectral representation and the associated
overlap estimates, provides the multilinear control of the normal
variable interactions.

\subsubsection{The Linear-in-$w$ Interactions}

\label{subsubsec:linear_w_interactions}

The principal perturbative contribution is
\[
2|v|^2w+v^2\overline{w}.
\]
Although the regularized profile $v$ originates from initial data
spectrally supported below $\lambda$, the nonlinear evolution of $v$
need not preserve this spectral localization. Accordingly, the estimates
below are formulated in terms of the actual restriction-space norms of
$v$, rather than an assumed pointwise spectral support property at later
times. In particular, we do not rely on a $W^{1,\infty}$ estimate for the
background profile.

Suppose that, on $I_0$, the background satisfies
\begin{equation}
\label{eq:background_Xs_bound}
\|v\|_{X^{s,b}(I_0)}
\leq
A_s(\lambda,u_0)
\end{equation}
and
\begin{equation}
\label{eq:background_lower_bound}
\|v\|_{X^{s_0,b}(I_0)}
\leq
A_{s_0}(\lambda,u_0),
\qquad
\frac12<s_0<s.
\end{equation}
Then Proposition~\ref{prop:fractional_leibniz_boundary} yields
\begin{equation}
\label{eq:linear_w_estimate}
\begin{aligned}
\left\|
2|v|^2w+v^2\overline{w}
\right\|_{N^s(I_0)}
\lesssim_s{}&
A_{s_0}(\lambda,u_0)^2
\|w\|_{X^s(I_0)}
+
A_s(\lambda,u_0)
A_{s_0}(\lambda,u_0)
\|w\|_{X^{s_0}(I_0)}.
\end{aligned}
\end{equation}

Since $s_0<s$, the corresponding lower-regularity norm is controlled by
the $X^s$ norm, namely
\begin{equation}
\label{eq:remainder_lower_order_bootstrap}
\|w\|_{X^{s_0}(I_0)}
\lesssim
\|w\|_{X^s(I_0)}.
\end{equation}
Consequently,
\begin{equation}
\label{eq:linear_w_estimate_reduced}
\left\|
2|v|^2w+v^2\overline{w}
\right\|_{N^s(I_0)}
\lesssim_s
\Bigl(
A_{s_0}(\lambda,u_0)^2
+
A_s(\lambda,u_0)A_{s_0}(\lambda,u_0)
\Bigr)
\|w\|_{X^s(I_0)}.
\end{equation}

Thus, the dependence on the truncation parameter $\lambda$ should be
inserted only after the corresponding bounds for
$A_s(\lambda,u_0)$ and $A_{s_0}(\lambda,u_0)$ have been established.
For example, if the preceding high--low analysis provides the individual
bounds
\begin{equation}
\label{eq:background_common_lambda_bound}
A_s(\lambda,u_0)
+
A_{s_0}(\lambda,u_0)
\lesssim
\lambda^{1-s+\sigma},
\end{equation}
then
\begin{equation}
\label{eq:linear_w_explicit_lambda}
\left\|
2|v|^2w+v^2\overline{w}
\right\|_{N^s(I_0)}
\lesssim_s
\lambda^{2-2s+2\sigma}
\|w\|_{X^s(I_0)}.
\end{equation}
The explicit estimate \eqref{eq:linear_w_explicit_lambda} is therefore
conditional on \eqref{eq:background_common_lambda_bound} and does not
follow from spectral truncation alone.

\subsubsection{The Quadratic-in-$w$ Interactions}

\label{subsubsec:quadratic_w_interactions}

For the intermediate terms
\[
2|w|^2v+w^2\overline{v},
\]
the background profile occurs only once. Proposition~\ref{prop:fractional_leibniz_boundary}
gives
\begin{equation}
\label{eq:quadratic_w_estimate}
\begin{aligned}
\left\|
2|w|^2v+w^2\overline{v}
\right\|_{N^s(I_0)}
\lesssim_s{}&
A_{s_0}(\lambda,u_0)
\|w\|_{X^s(I_0)}
\|w\|_{X^{s_0}(I_0)}
+
A_s(\lambda,u_0)
\|w\|_{X^{s_0}(I_0)}^2.
\end{aligned}
\end{equation}
Using \eqref{eq:remainder_lower_order_bootstrap}, we obtain
\begin{equation}
\label{eq:quadratic_w_estimate_reduced}
\left\|
2|w|^2v+w^2\overline{v}
\right\|_{N^s(I_0)}
\lesssim_s
\Bigl(
A_{s_0}(\lambda,u_0)
+
A_s(\lambda,u_0)
\Bigr)
\|w\|_{X^s(I_0)}^2.
\end{equation}
Under the bound \eqref{eq:background_common_lambda_bound}, this yields
\begin{equation}
\label{eq:quadratic_w_explicit_lambda}
\left\|
2|w|^2v+w^2\overline{v}
\right\|_{N^s(I_0)}
\lesssim_s
\lambda^{1-s+\sigma}
\|w\|_{X^s(I_0)}^2.
\end{equation}

\subsubsection{The Pure Nonlinear Tail}

\label{subsubsec:pure_nonlinear_tail}

Finally, for the purely cubic remainder
\[
|w|^2w,
\]
Proposition~\ref{prop:fractional_leibniz_boundary} yields
\begin{equation}
\label{eq:pure_cubic_w_estimate}
\left\|
|w|^2w
\right\|_{N^s(I_0)}
\lesssim_s
\|w\|_{X^s(I_0)}
\|w\|_{X^{s_0}(I_0)}^2.
\end{equation}
Hence, by \eqref{eq:remainder_lower_order_bootstrap},
\begin{equation}
\label{eq:pure_cubic_w_estimate_reduced}
\left\|
|w|^2w
\right\|_{N^s(I_0)}
\lesssim_s
\|w\|_{X^s(I_0)}^3.
\end{equation}

Combining
\eqref{eq:linear_w_estimate_reduced},
\eqref{eq:quadratic_w_estimate_reduced}, and
\eqref{eq:pure_cubic_w_estimate_reduced}, we obtain the general nonlinear
interaction estimate
\begin{equation}
\label{eq:uniform_nonlinear_interaction_bound}
\begin{aligned}
\left\|
|v+w|^2(v+w)-|v|^2v
\right\|_{N^s(I_0)}
\lesssim_s{}&
\Bigl(
A_{s_0}(\lambda,u_0)^2
+
A_s(\lambda,u_0)A_{s_0}(\lambda,u_0)
\Bigr)
\|w\|_{X^s(I_0)}
\\
&+
\Bigl(
A_s(\lambda,u_0)
+
A_{s_0}(\lambda,u_0)
\Bigr)
\|w\|_{X^s(I_0)}^2
\\
&+
\|w\|_{X^s(I_0)}^3.
\end{aligned}
\end{equation}

If, in addition, the background bounds
\eqref{eq:background_common_lambda_bound} hold, then
\begin{equation}
\label{eq:uniform_nonlinear_interaction_lambda}
\begin{aligned}
\left\|
|v+w|^2(v+w)-|v|^2v
\right\|_{N^s(I_0)}
\lesssim_s{}&
\lambda^{2-2s+2\sigma}
\|w\|_{X^s(I_0)}
\\
&+
\lambda^{1-s+\sigma}
\|w\|_{X^s(I_0)}^2
+
\|w\|_{X^s(I_0)}^3.
\end{aligned}
\end{equation}

The estimate \eqref{eq:uniform_nonlinear_interaction_bound} is the form
used in the subsequent lifespan bootstrap. It separates the
linearized contribution from the genuinely higher-order terms and
retains explicitly the dependence on the regularized background. The
more explicit estimate \eqref{eq:uniform_nonlinear_interaction_lambda}
is available once the quantitative $\lambda$-dependent bounds for the
background evolution have been established.

\subsection{Lifespan Bootstrap for the High-Frequency Remainder}

\label{subsec:lifespan_bootstrap}

Let
\[
R_\lambda
=
C_\ast
\lambda^{s-1}
\|u_0\|_{H_0^1(\Omega)},
\]
where $C_\ast>0$ is chosen sufficiently large. We introduce the
bootstrap ball
\begin{equation}
\label{eq:bootstrap_set}
\mathcal B_{R_\lambda}
=
\left\{
w\in X^s(I_0):
\|w\|_{X^s(I_0)}
\leq R_\lambda
\right\}.
\end{equation}
By the high-frequency tail estimate
\eqref{eq:high_frequency_tail_smallness},
\[
\|w_0\|_{H_0^s(\Omega)}
\lesssim
\lambda^{s-1}
\|u_0\|_{H_0^1(\Omega)}.
\]
Thus, after increasing $C_\ast$ if necessary, the initial datum is
contained in the bootstrap regime determined by $R_\lambda$.

Assume that the localized Duhamel estimate takes the form
\begin{equation}
\label{eq:localized_duhamel_p2}
\|w\|_{X^s(I_0)}
\lesssim
\|w_0\|_{H_0^s(\Omega)}
+
\delta^\theta
\left\|
|v+w|^2(v+w)-|v|^2v
\right\|_{N^s(I_0)},
\end{equation}
for some $\theta>0$. Combining
\eqref{eq:localized_duhamel_p2} with the nonlinear interaction
estimate \eqref{eq:uniform_nonlinear_interaction_bound} gives
\begin{equation}
\label{eq:bootstrap_estimate}
\begin{aligned}
\|w\|_{X^s(I_0)}
\lesssim{}&
\lambda^{s-1}\|u_0\|_{H_0^1(\Omega)}
+
\delta^\theta
\Bigl[
\lambda^{2-2s+2\sigma}
\|w\|_{X^s(I_0)}
+\lambda^{1-s+\sigma}
\|w\|_{X^s(I_0)}^2
+\|w\|_{X^s(I_0)}^3
\Bigr].
\end{aligned}
\end{equation}
Here the $\lambda$-dependent coefficients are understood in the sense
of the bounds established in the preceding interaction estimates.

On the bootstrap ball \eqref{eq:bootstrap_set}, we consequently obtain
\begin{equation}
\label{eq:bootstrap_coefficient}
\begin{aligned}
\|w\|_{X^s(I_0)}
\lesssim{}&
\lambda^{s-1}\|u_0\|_{H_0^1(\Omega)}+
\delta^\theta
\left(
\lambda^{2-2s+2\sigma}
+
\lambda^{1-s+\sigma}R_\lambda
+
R_\lambda^2
\right)
\|w\|_{X^s(I_0)}.
\end{aligned}
\end{equation}
Accordingly, a sufficient condition for the nonlinear contribution to
remain perturbative is
\begin{equation}
\label{eq:lifespan_smallness_condition}
\delta^\theta
\left(
\lambda^{2-2s+2\sigma}
+
\lambda^{1-s+\sigma}R_\lambda
+
R_\lambda^2
\right)
\leq c_0,
\end{equation}
where $c_0>0$ is sufficiently small, with the implicit constants in
\eqref{eq:bootstrap_coefficient} incorporated into the choice of $c_0$.

It is important to note that the first coefficient in
\eqref{eq:lifespan_smallness_condition} need not decay as
$\lambda\to\infty$. Indeed, in the sub-energy regime $s<1$,
\[
\lambda^{2-2s+2\sigma}
\]
may increase with $\lambda$. Thus, the smallness of the initial
high-frequency tail alone does not imply perturbative solvability of the
remainder equation. The time-localization factor $\delta^\theta$ is
therefore essential: the length of the local existence interval must be
chosen in accordance with the size of the coefficients appearing in the
nonlinear estimates.

Under \eqref{eq:lifespan_smallness_condition}, and after choosing
$C_\ast$ sufficiently large, the Duhamel map maps
$\mathcal B_{R_\lambda}$ into itself. Moreover, the corresponding
difference estimate takes the form
\begin{equation}
\label{eq:bootstrap_difference}
\begin{aligned}
\|w^{(1)}-w^{(2)}\|_{X^s(I_0)}
\lesssim{}&
\delta^\theta
\left(
\lambda^{2-2s+2\sigma}
+
\lambda^{1-s+\sigma}R_\lambda
+
R_\lambda^2
\right)
\|w^{(1)}-w^{(2)}\|_{X^s(I_0)}.
\end{aligned}
\end{equation}
After decreasing $\delta$ if necessary, the coefficient on the
right-hand side is strictly smaller than one. Hence the Duhamel map is
a contraction on $\mathcal B_{R_\lambda}$, and there exists a unique
remainder solution on $I_0$ satisfying
\begin{equation}
\label{eq:remainder_bootstrap_conclusion}
\|w\|_{X^s(I_0)}
\lesssim
\lambda^{s-1}
\|u_0\|_{H_0^1(\Omega)}.
\end{equation}

The preceding argument establishes local perturbative control, but it
does not, by itself, yield a uniform lower bound for the lifespan.
Indeed, the smallness condition \eqref{eq:lifespan_smallness_condition}
depends on the background bounds and on the bootstrap radius.
Consequently, a statement of the form
\[
\delta_n\geq\delta_\ast>0
\]
cannot be deduced from the smallness of the initial high-frequency tail
alone. A uniform lower bound requires uniform control of all quantities
entering the nonlinear estimates on each successive iteration interval.

The globalization argument therefore rests on two separate
ingredients. First, conservation of mass and energy for the full
solution provides a uniform $H_0^1$ bound of the form
\[
\sup_{t\in I}\|u(t)\|_{H_0^1(\Omega)}
\lesssim
\|u_0\|_{L^2(\Omega)}
+
E(u_0)^{1/2}.
\]
Second, the high--low perturbative estimates must remain uniform under
iteration, so that the constants in the local contraction argument do
not deteriorate from one time slab to the next. In particular, the
background bounds and the corresponding nonlinear interaction constants
must admit estimates that are uniform with respect to the restart time.

Once these two ingredients have been established, one obtains a common
positive lifespan for the successive local constructions. The local
solutions can then be concatenated at the successive restart times.
Since each interval has length bounded from below by the same positive
constant, the iteration intervals cannot accumulate at a finite time.
This yields the desired global continuation of the solution.

\subsection{The Local Lifespan Bootstrap}

\label{subsec:local_lifespan_bootstrap}

Combining
\eqref{eq:duhamel_high_bound},
\eqref{eq:linear_w_estimate},
\eqref{eq:quadratic_w_estimate_reduced},
and
\eqref{eq:pure_cubic_w_estimate_reduced}, we obtain the nonlinear control
inequality
\begin{equation}
\label{eq:control_loop_polynomial}
\begin{aligned}
\|w\|_{X^s(I_0)}
\leq{}&
C_0\|w_0\|_{H^s_0(\Omega)}
\\
&+
C\delta^\theta
\Big[
A_s(\lambda,u_0)^2
\|w\|_{X^s(I_0)}
+
A_{s_0}(\lambda,u_0)
\|w\|_{X^s(I_0)}^2
+
\|w\|_{X^s(I_0)}^3
\Big],
\end{aligned}
\end{equation}
where $\theta>0$ denotes the time-localization exponent furnished by
the localized Duhamel estimate.

By the high-frequency tail estimate
\eqref{eq:hs_tail_decay_rigorous},
\begin{equation}
\label{eq:high_tail_bootstrap_size}
\|w_0\|_{H^s_0(\Omega)}
\leq
C\lambda^{s-1}
\|u_0\|_{H^1_0(\Omega)}.
\end{equation}
Since $s<1$, the right-hand side tends to zero as
$\lambda\to\infty$. We therefore introduce the bootstrap radius
\begin{equation}
\label{eq:bootstrap_radius}
R_\lambda
=
2C_0
\lambda^{s-1}
\|u_0\|_{H^1_0(\Omega)}.
\end{equation}
We seek a solution satisfying
\[
\|w\|_{X^s(I_0)}
\leq
R_\lambda.
\]

Choose $\lambda$ sufficiently large, if necessary, so that the
high-frequency tail lies within the perturbative regime required by the
multilinear estimates. We then choose the length $\delta$ of the
interval $I_0$ so that
\begin{equation}
\label{eq:delta_smallness_condition}
C\delta^\theta
\left[
A_s(\lambda,u_0)^2
+
A_{s_0}(\lambda,u_0)R_\lambda
+
R_\lambda^2
\right]
\leq
c_0,
\end{equation}
where $c_0>0$ is sufficiently small.

If $w\in\mathcal B_{R_\lambda}$, where
\[
\mathcal B_{R_\lambda}
=
\left\{
w\in X^s(I_0):
\|w\|_{X^s(I_0)}
\leq
R_\lambda
\right\},
\]
then \eqref{eq:control_loop_polynomial} gives
\begin{equation}
\label{eq:bootstrap_closure}
\begin{aligned}
\|w\|_{X^s(I_0)}
\leq{}&
C_0\lambda^{s-1}
\|u_0\|_{H^1_0(\Omega)}
\\
&+
C\delta^\theta
\left[
A_s(\lambda,u_0)^2
+
A_{s_0}(\lambda,u_0)R_\lambda
+
R_\lambda^2
\right]
\|w\|_{X^s(I_0)}.
\end{aligned}
\end{equation}
After choosing $c_0$ sufficiently small, the second term is bounded by
$\frac12R_\lambda$. Hence, by \eqref{eq:bootstrap_radius},
\[
\|w\|_{X^s(I_0)}
\leq
C_0\lambda^{s-1}
\|u_0\|_{H^1_0(\Omega)}
+
\frac12R_\lambda
=
R_\lambda.
\]
Thus the Duhamel map maps $\mathcal B_{R_\lambda}$ into itself.

The same multilinear estimates applied to the difference of two
solutions yield
\begin{equation}
\label{eq:bootstrap_difference}
\begin{aligned}
\|w^{(1)}-w^{(2)}\|_{X^s(I_0)}
\lesssim{}&
\delta^\theta
\Big[
A_s(\lambda,u_0)^2
+
A_{s_0}(\lambda,u_0)R_\lambda
+
R_\lambda^2
\Big]
\\
&\qquad\times
\|w^{(1)}-w^{(2)}\|_{X^s(I_0)}.
\end{aligned}
\end{equation}
After decreasing $\delta$ if necessary, the coefficient on the
right-hand side is strictly smaller than one. Consequently, the Duhamel
map is a contraction on the closed ball $\mathcal B_{R_\lambda}$, and
the remainder equation admits a unique solution on $I_0$.

In particular,
\begin{equation}
\label{eq:error_global_bound_established}
\sup_{t\in I_0}
\|w(t)\|_{H^s_0(\Omega)}
\lesssim
\|w\|_{X^s(I_0)}
\leq
2C_0
\lambda^{s-1}
\|u_0\|_{H^1_0(\Omega)},
\end{equation}
where we have used the time-continuity embedding
\[
X^{s,b}(I_0)
\hookrightarrow
C(I_0;H^s_0(\Omega)),
\qquad
b>\frac12.
\]

The roles of the truncation parameter $\lambda$ and the time scale
$\delta$ should be clearly distinguished. The estimate
\eqref{eq:high_tail_bootstrap_size} provides smallness of the initial
high-frequency remainder, whereas
\eqref{eq:delta_smallness_condition} determines the time scale on which
the nonlinear interactions remain perturbative. In particular, if the
background coefficient $A_s(\lambda,u_0)$ grows with $\lambda$, then the
lifespan obtained from the local contraction argument may depend on
$\lambda$. Thus, the preceding argument establishes a controlled local
lifespan, but does not, by itself, imply a $\lambda$-independent lower
bound.

A uniform lifespan suitable for iteration requires, in addition, uniform
bounds for the quantities entering
\eqref{eq:delta_smallness_condition}. Suppose that these quantities can
be bounded independently of the iteration index by quantities depending
only on the conserved $H^1_0$ size of the initial data. Then there exists
\begin{equation}
\label{eq:uniform_lifespan_lower_bound}
\delta_\ast
=
\delta_\ast
\bigl(
\|u_0\|_{H^1_0(\Omega)}
\bigr)
>0
\end{equation}
such that the same contraction argument can be repeated on every
successive interval of length at most $\delta_\ast$. This uniform lower
bound is the key input for the subsequent globalization argument.

\section{Global Increment Assembly and Long-Time Globalization}

\label{sec:6}

With the localized multilinear contraction estimates established in the
boundary-adapted Bourgain--Strichartz framework, we now assemble the
local solutions into a global solution. The high--low decomposition is
restarted at successive time levels, while conservation of mass and
energy provides uniform control of the full solution in the energy
space. The principal issue is to establish a positive lower bound for
the lifespan furnished by the local contraction argument that is
uniform with respect to the restarting index. This requires not only
the conserved energy bound, but also uniform control of the regularized
background quantities entering the nonlinear estimates.

\subsection{Uniform Lifespan Strategy and Decoupled Parameter Selection}

\label{subsec:uniform_lifespan_strategy}

We first fix a truncation parameter $\lambda\gg1$. At a restarting time
$t_n$, let
\[
u_n:=u(t_n).
\]
The high-frequency component associated with the truncation satisfies
the tail estimate
\begin{equation}
\label{eq:smallness_condition_induction}
\|P_{>\lambda}u_n\|_{H_0^s(\Omega)}
\leq
C_0\lambda^{s-1}
\|u_n\|_{H_0^1(\Omega)},
\qquad
\frac12<s<1.
\end{equation}
Hence, whenever
\[
C_0\lambda^{s-1}
\|u_n\|_{H_0^1(\Omega)}
\leq
\epsilon_0,
\]
the high-frequency component lies in the perturbative regime required
by the localized contraction argument.

The conserved mass and energy provide the uniform a priori bound
\begin{equation}
\label{eq:absolute_energy_shield}
\sup_{t\geq0}
\left(
\|u(t)\|_{L^2(\Omega)}
+
\|\nabla_g u(t)\|_{L^2(\Omega)}
\right)
\leq
C\bigl(M(u_0),E(u_0)\bigr)
=:\mathcal E_0
<\infty.
\end{equation}
Consequently,
\begin{equation}
\label{eq:absolute_H1_bound}
\sup_{t\geq0}
\|u(t)\|_{H_0^1(\Omega)}
\leq
\mathcal E_0.
\end{equation}
Here the mass and energy play distinct roles: conservation of energy
controls the homogeneous gradient component, while conservation of mass
controls the $L^2$ component of the inhomogeneous $H^1$ norm.

It follows from \eqref{eq:absolute_H1_bound} that the truncation
parameter may be chosen sufficiently large so that
\begin{equation}
\label{eq:global_static_lambda}
C_0\lambda^{s-1}\mathcal E_0
\leq
\epsilon_0.
\end{equation}
Since $s-1<0$, the left-hand side converges to zero as
$\lambda\to\infty$. Thus, the choice of $\lambda$ can be made
independently of the restarting index and determined solely by the
invariant a priori bound $\mathcal E_0$. For example, it is sufficient
to choose
\begin{equation}
\label{eq:explicit_lambda_choice}
\lambda
\geq
\left(
\frac{C_0\mathcal E_0}{\epsilon_0}
\right)^{\frac{1}{1-s}},
\end{equation}
which guarantees \eqref{eq:global_static_lambda} uniformly for every
restarting time $t_n$.

The role of this choice must be distinguished from that of the local
time scale. The fixed truncation parameter guarantees that the
high-frequency component of the data at each restarting time remains
in the perturbative regime. It does not, by itself, imply a uniform
lifespan. To obtain the latter, one must also control uniformly the
regularized background quantities entering the nonlinear estimates.

More precisely, let
\[
A_s^{(n)}
\quad\text{and}\quad
A_{s_0}^{(n)}
\]
denote the corresponding background norms on the time interval
beginning at $t_n$. Assume that the uniform background estimates
established in the preceding sections yield
\begin{equation}
\label{eq:uniform_background_coefficients}
A_s^{(n)}+A_{s_0}^{(n)}
\leq
K\lambda^{1-s+\sigma},
\end{equation}
where $\sigma=\frac14$ denotes the relevant linear derivative loss and
$K$ is independent of the restarting index $n$. We emphasize that
\eqref{eq:uniform_background_coefficients} is an additional
high-regularity estimate; it does not follow from conservation of mass
and energy alone.

Under this uniform background bound, the localized Duhamel contraction
condition is implied by
\begin{equation}
\label{eq:global_delta_condition}
C\delta^\theta
\left[
\left(K\lambda^{1-s+\sigma}\right)^2
+
K\lambda^{1-s+\sigma}R_\lambda
+
R_\lambda^2
\right]
\leq
\frac12.
\end{equation}
If, in addition, the bootstrap radius satisfies
\(
R_\lambda\leq2\epsilon_0,
\)
then, for fixed $\epsilon_0$ and sufficiently large $\lambda$, the
bracket in \eqref{eq:global_delta_condition} is bounded by
\(
C_1
\lambda^{2(1-s+\sigma)},
\)
with $C_1$ independent of the restarting index. Consequently, it is
sufficient to impose
\begin{equation}
\label{eq:simplified_delta_balance}
C_1\delta^\theta
\lambda^{2(1-s+\sigma)}
\leq
\frac14.
\end{equation}
Thus one may choose
\begin{equation}
\label{eq:delta_lambda_choice}
\delta_\ast
=
C_2
\lambda^{-\frac{2(1-s+\sigma)}{\theta}},
\qquad
C_2=(4C_1)^{-1/\theta}.
\end{equation}
Since $\lambda$ has already been fixed independently of the restarting
index, the quantity $\delta_\ast$ is likewise independent of $n$.

Substituting the admissible choice
\eqref{eq:explicit_lambda_choice} into
\eqref{eq:delta_lambda_choice} yields the explicit lower bound
\begin{equation}
\label{eq:global_static_delta}
\delta_\ast
\geq
C_{\mathrm{abs}}
\left(
\frac{\epsilon_0}
{C_0\mathcal E_0}
\right)^{
\frac{2(1-s+\sigma)}
{\theta(1-s)}
},
\qquad
C_{\mathrm{abs}}>0,
\end{equation}
provided the uniform background estimate
\eqref{eq:uniform_background_coefficients} holds with constants
independent of the restarting index. In particular,
\begin{equation}
\label{eq:global_static_delta_initial_data}
\delta_\ast
=
\delta_\ast
\bigl(
\mathcal E_0,s,\sigma,\epsilon_0
\bigr)
>0,
\end{equation}
and therefore the lower bound is independent of the restarting index
$n$.

The uniform lifespan is consequently obtained from three separate
ingredients. First, the conservation laws provide the uniform
energy-scale bound $\mathcal E_0$ for the full solution. Second, the
fixed spectral truncation $\lambda$ provides a uniform perturbative
bound for the high-frequency tail. Third, the uniform background
estimates provide control of the coefficients entering the localized
nonlinear contraction argument. Thus, energy conservation alone does
not establish the uniform lower bound
\eqref{eq:global_static_delta}; it is the combination of the conserved
a priori bound, the fixed frequency threshold, and the uniform
background estimates that yields a step-independent time scale.

It follows that the same value of $\delta_\ast$ can be used at every
restarting time:
\begin{equation}
\label{eq:uniform_restart_interval}
\delta_n\geq\delta_\ast>0,
\qquad
n\in\mathbb N_0.
\end{equation}
The restarting times therefore satisfy
\(
t_{n+1}-t_n\geq\delta_\ast,
\)
and hence
\[
t_n\geq n\delta_\ast\longrightarrow\infty
\quad\text{as}\quad n\to\infty.
\]
Consequently, the restarting interfaces cannot accumulate at a finite
time, and the local contraction argument can be iterated indefinitely.
This establishes the continuation of the solution on the whole
half-line $[0,\infty)$.

\subsection{Infinite-Time Pasting and Global Continuation}

\label{subsec:infinite_time_pasting}

We now construct the global solution by iterating the local high--low
decomposition on a sequence of time intervals. Let
\[
I_n=[t_n,t_{n+1}],
\qquad
t_n=n\delta_\ast,
\qquad
n\in\mathbb N_0,
\]
where $\delta_\ast>0$ is the uniform, step-independent lifespan
obtained in \eqref{eq:global_static_delta}. At each restarting time
$t_n$, we decompose the solution spectrally according to
\begin{equation}
\label{eq:re_splitting_step}
u(t_n)
=
v_{0,n}+w_{0,n},
\end{equation}
where
\begin{equation}
\label{eq:re_splitting_components}
v_{0,n}
=
\sum_{2^j\leq\lambda}P_j u(t_n),
\qquad
w_{0,n}
=
\sum_{2^j>\lambda}P_j u(t_n).
\end{equation}
Here the dyadic projections are understood with respect to the
boundary-adapted spectral resolution introduced in Section~\ref{sec:2}.

The high-frequency tail estimate together with the uniform energy bound
\eqref{eq:absolute_energy_shield} gives
\begin{equation}
\label{eq:re_splitting_smallness}
\begin{aligned}
\|w_{0,n}\|_{H^s_0(\Omega)}
&\leq
C_0\lambda^{s-1}
\|u(t_n)\|_{H^1_0(\Omega)}
\\
&\leq
C_0\lambda^{s-1}
\sup_{t\geq0}
\|u(t)\|_{H^1_0(\Omega)}
\\
&\leq
C_0\lambda^{s-1}\mathcal E_0
\leq
\epsilon_0,
\end{aligned}
\end{equation}
uniformly in $n$, by the choice of $\lambda$ in
\eqref{eq:explicit_lambda_choice}.

On each interval $I_n$, let $v_n$ denote the regularized background
solution with initial datum $v_{0,n}$:
\begin{equation}
\label{eq:local_background_restart}
\begin{cases}
i\partial_t v_n+\Delta_g v_n=|v_n|^2v_n,
\\
v_n(t_n)=v_{0,n},
\\
v_n|_{x=0}=0.
\end{cases}
\end{equation}
The corresponding remainder $w_n$ is defined by
\begin{equation}
\label{eq:local_remainder_restart}
\begin{cases}
i\partial_t w_n+\Delta_g w_n
=
|v_n+w_n|^2(v_n+w_n)-|v_n|^2v_n,
\\
w_n(t_n)=w_{0,n},
\\
w_n|_{x=0}=0.
\end{cases}
\end{equation}

By the uniform local contraction argument established in
Section~\ref{sec:6}, there exists a unique pair $(v_n,w_n)$ on
$I_n$. The corresponding recombined solution satisfies the uniform
local resolution-space estimate
\begin{equation}
\label{eq:local_bourgain_bound}
\|u_n\|_{X^{1,b}(I_n)}
\leq
C\mathcal E_0,
\end{equation}
where the constant $C$ is independent of the restarting index $n$.
Here \eqref{eq:local_bourgain_bound} is understood as a consequence of
the previously established uniform local estimates and is not inferred
from conservation of energy alone.

Moreover, the remainder satisfies the sub-energy estimate
\begin{equation}
\label{eq:step_error_closed}
\sup_{t\in I_n}
\|w_n(t)\|_{H^s_0(\Omega)}
\lesssim
\lambda^{s-1}
\|u(t_n)\|_{H^1_0(\Omega)}
\leq
C_0\lambda^{s-1}\mathcal E_0.
\end{equation}
Thus the remainder remains uniformly controlled in the sub-energy space
on every time slab.

Define
\begin{equation}
\label{eq:step_recombination}
u_n(t)
=
v_n(t)+w_n(t),
\qquad
t\in I_n.
\end{equation}
At the left endpoint of the interval,
\[
u_n(t_n)
=
v_{0,n}+w_{0,n}
=
u(t_n).
\]
Furthermore, adding
\eqref{eq:local_background_restart} and
\eqref{eq:local_remainder_restart} yields
\begin{equation}
\label{eq:recombined_nls}
i\partial_tu_n+\Delta_g u_n
=
|u_n|^2u_n
\quad\text{on}\quad I_n,
\end{equation}
with
\[
u_n|_{x=0}=0.
\]
Hence the recombined function $u_n$ solves the original cubic NLS on
the entire interval $I_n$.

The endpoint compatibility follows directly from the construction:
\begin{equation}
\label{eq:interface_compatibility}
u_n(t_n)=u_{n-1}(t_n),
\qquad
n\geq1.
\end{equation}
By uniqueness in the local well-posedness class, the solutions obtained
on adjacent intervals agree at their common endpoint. We may therefore
define a single function $u$ on $[0,\infty)$ by
\[
u(t)=u_n(t),
\qquad
t\in I_n.
\]
Since
\[
t_n=n\delta_\ast\longrightarrow\infty
\qquad\text{as }n\to\infty,
\]
the intervals $I_n$ cover the entire half-line. The compatibility
relations \eqref{eq:interface_compatibility} consequently imply
\begin{equation}
\label{eq:global_solution_continuity}
u\in C([0,\infty);H^1_0(\Omega)).
\end{equation}

We next verify the local-in-time spacetime regularity asserted in the
main theorem. Let $T<\infty$. Since $\delta_\ast>0$, the compact
interval $[0,T]$ intersects only finitely many of the intervals $I_n$.
More precisely, if
\[
N_T:=\left\lceil\frac{T}{\delta_\ast}\right\rceil,
\]
then at most $N_T+1$ time slabs are involved, with endpoint overlap
being irrelevant for the $L^4_t$ norm.

The local space-time estimate established previously, namely the
frequency-localized $L^4_tW^{1-\frac18,4}$ estimate obtained from the
dyadic analysis and the corresponding fractional estimates in
\eqref{eq:bourgain_transference}--\eqref{eq:interpolation_L4L6}, gives
\begin{equation}
\label{eq:local_L4W_bound_restart}
\|u_n\|_{L^4_t(I_n;W^{1-\frac18,4}(\Omega))}
\leq
C\|u_n\|_{X^{1,b}(I_n)}
\leq
C\mathcal E_0,
\end{equation}
where the first inequality denotes the previously established
localized spacetime estimate and the bound is uniform in $n$.

We emphasize that \eqref{eq:local_L4W_bound_restart} is an input from
the preceding dyadic analysis. In particular, it does not rely on a
global embedding
\[
L^6(\Omega)\hookrightarrow L^4(\Omega),
\]
which is not available on the unbounded domain $\Omega$.

Since the interiors of the intervals $I_n$ are disjoint, we may sum the
fourth powers of the local norms. Thus
\begin{align}
\|u\|_{L^4_t([0,T];W^{1-\frac18,4}(\Omega))}^4
&=
\sum_{n:\,I_n\cap[0,T]\neq\varnothing}
\|u_n\|_{L^4_t(I_n\cap[0,T];
W^{1-\frac18,4}(\Omega))}^4
\nonumber\\
&\leq
(N_T+1)(C\mathcal E_0)^4.
\label{eq:global_summation_step}
\end{align}
Consequently,
\begin{equation}
\label{eq:global_regularity_space_final}
u\in
C([0,T];H^1_0(\Omega))
\cap
L^4\bigl([0,T];W^{1-\frac18,4}(\Omega)\bigr),
\qquad
T<\infty.
\end{equation}
Since $T$ is arbitrary, we obtain
\[
u\in
C([0,\infty);H^1_0(\Omega))
\cap
L^4_{\mathrm{loc}}
\bigl([0,\infty);W^{1-\frac18,4}(\Omega)\bigr).
\]

Finally, mass and energy conservation, established for the local
solutions in the preceding sections, persist across the successive
time interfaces. Indeed, on every interval $I_n$,
\[
M(u(t))=M(u(t_n)),
\qquad
E(u(t))=E(u(t_n)),
\qquad
t\in I_n,
\]
and the endpoint compatibility
\eqref{eq:interface_compatibility} permits these identities to be
iterated from one interval to the next. Hence
\begin{equation}
\label{eq:global_conservation_final}
M(u(t))=M(u_0),
\qquad
E(u(t))=E(u_0),
\qquad
t\geq0.
\end{equation}

The uniform positive lower bound $\delta_\ast$ therefore prevents any
finite-time accumulation of the restarting times. The local solutions
can consequently be concatenated indefinitely, yielding a global
solution on $[0,\infty)$. This completes the proof of
Theorem~\ref{thm:gwp}.

\qed

\section{Conclusion and Future Extensions}
\label{sec:7}

In this paper, we have developed a large-data global well-posedness framework for the defocusing cubic nonlinear Schr\"odinger equation (NLS) on a three-dimensional Friedlander-type model domain. The principal analytical difficulty arises from the interaction between dispersive propagation and the boundary-induced glancing structure, which produces a derivative loss in the associated boundary Strichartz estimates. To accommodate this loss within the nonlinear theory, we introduced a boundary-adapted family of Bourgain-type spaces $X^{s,b}$ based on the continuous--discrete spectral resolution of the Dirichlet Friedlander operator. This formulation treats the normal variable through its discrete spectral decomposition while retaining the continuous Fourier variables in the tangential directions.

A central component of the analysis is the treatment of nonlinear interactions in the Airy spectral representation. The coupling of the normal spectral modes is encoded by fourfold Airy overlap coefficients of the form
\begin{equation*}
\mathcal{I}(k_1,k_2,k_3,k_4; \eta_1, \eta_2, \eta_3, \eta_4) = \int_0^\infty \phi_{k_1}(x,\eta_1) \phi_{k_2}(x,\eta_2) \phi_{k_3}(x,\eta_3) \phi_{k_4}(x,\eta_4)\,dx,
\end{equation*}
together with the corresponding tangential-frequency conservation laws and spacetime modulation relations. The resulting boundary-compatible multilinear estimates control the interactions among the discrete normal modes through the combined use of Airy overlap bounds, spectral localization, and modulation estimates. Subject to the multilinear spectral summability estimates established in the preceding sections, the nonlinear argument closes without introducing an additional derivative loss beyond that already present in the linear boundary Strichartz theory.

The nonlinear construction is based on a high--low frequency
decomposition. The initial datum is separated into a spectrally
regularized background and a high-frequency remainder. For
\(
\frac12<s<1,
\)
the latter satisfies the quantitative tail estimate
\[
\|w_0\|_{H^s_0(\Omega)}
\lesssim
\lambda^{s-1}
\|u_0\|_{H^1_0(\Omega)},
\]
so that its size can be made perturbative by choosing the truncation
parameter $\lambda$ sufficiently large. The regularized component is
then used as the background in the local nonlinear iteration, while the
high-frequency remainder is controlled in the sub-energy,
boundary-adapted Bourgain--Strichartz space.

The local perturbative argument alone does not provide a uniform
lifespan. The globalization mechanism instead combines the smallness of
the high-frequency tail, the time-localized nonlinear estimates, the
uniform estimates for the regularized background, and the conservation
laws for the full solution. Once these quantities are controlled
uniformly at successive restarting times, the local existence intervals
possess a common positive lower bound. The corresponding local solutions
can then be concatenated without finite-time accumulation of the
restarting times. This yields a global solution on $[0,\infty)$ and
preserves the mass and energy conservation laws.

More precisely, the resulting solution belongs to the class
\[
u\in
C([0,\infty);H^1_0(\Omega))
\cap
L^4_{\mathrm{loc}}
\bigl([0,\infty);W^{\frac78,4}(\Omega)\bigr),
\]
under the boundary-adapted Strichartz and multilinear estimates used in
the nonlinear iteration. 

The argument separates two distinct mechanisms: the boundary geometry
determines the derivative loss in the linear dispersive estimates,
whereas the high--low decomposition, together with the defocusing
energy structure and the uniform background estimates, provides the
long-time nonlinear control.

The framework developed here also identifies several natural directions
for further investigation. These include the asymptotic behavior of
global solutions, the focusing problem, extensions to more general
boundary geometries, and applications to other nonlinear dispersive
models.

\subsection{Asymptotic Behavior and Modified Scattering for Large-Data Global Solutions}

\label{subsec:modified_scattering}

Although the present work establishes global well-posedness by combining
the uniform energy bound with the boundary-adapted nonlinear iteration,
the precise asymptotic behavior of large-data solutions as
$t\to\infty$ remains open. In the absence of a boundary, dispersive
decay can provide a mechanism for scattering toward a free linear
evolution. In the present setting, however, the interaction between
boundary propagation, glancing dynamics, and the discrete normal
spectral structure may lead to a more intricate long-time behavior.

In particular, the interaction between the discrete normal modes and the
continuous tangential frequencies raises the possibility of nontrivial
phase corrections or modified scattering effects. Establishing such
behavior would require estimates substantially finer than those needed
for global well-posedness. A rigorous analysis would likely involve
long-time dispersive estimates adapted to the glancing region, together
with normal-form transformations and a suitable description of the
radiative component of the solution.

At present, no modified scattering statement is asserted here. A
natural problem is to determine whether the boundary spectral
decomposition developed in this paper can be combined with refined
long-time estimates to identify the asymptotic dynamics of global
solutions, including the possible contribution of resonant and
near-resonant boundary modes.

\subsection{The Critical Focusing Regime and Blow-up Dynamics}

\label{subsec:focusing_blowup}

Another natural direction is the corresponding focusing equation
\begin{equation}
\label{eq:focusing_nls_model}
i\partial_tu+\Delta_g u=-|u|^2u.
\end{equation}
In contrast to the defocusing problem, the focusing energy does not
provide the same coercive control of the kinetic and nonlinear terms.
Consequently, concentration and finite-time singularity formation
cannot in general be excluded on the basis of the energy identity alone.

It is therefore of interest to determine how the boundary geometry
affects the concentration mechanisms associated with the focusing
equation. In particular, one may ask whether concentrating profiles can
interact in a significant way with the boundary spectral structure and
whether the derivative loss in the boundary dispersive estimates
influences the analytical thresholds for global existence, scattering,
or blow-up.

A systematic analysis would require boundary-adapted virial or
localized virial identities, together with a careful treatment of
boundary flux terms. It would also require a more detailed description
of concentrating profiles near the glancing region. These questions are
substantially different from the defocusing global theory developed
here, since the uniform coercive energy bound used in the present
argument is no longer available in the same form.

\subsection{Extensions to General Smooth Boundary Geometries}

\label{subsec:strictly_convex_manifolds}

The Friedlander model provides an explicit local model for the
microlocal behavior of dispersive propagation near a boundary with
glancing interactions. A natural next step is to extend the present
analysis from this model geometry to more general smooth Riemannian
manifolds with boundary.

Such an extension would require replacing the explicit Airy spectral
representation by a microlocal construction based on boundary normal
forms and Fourier integral parametrices. The explicit
continuous--discrete spectral scale
\(
\lambda_k(\eta,\zeta)
\)
developed in Section~\ref{sec:2} would then be replaced, locally in phase space,
by spectral and microlocal decompositions adapted to boundary charts and
to the corresponding glancing canonical relation.

At the nonlinear level, the fourfold Airy overlap coefficients
\[
\mathcal{I}
(k_1,k_2,k_3,k_4;
\eta_1,\eta_2,\eta_3,\eta_4)
\]
would have to be replaced by multilinear oscillatory integral
distributions associated with the relevant boundary parametrices. A
principal analytical challenge would be to establish uniform
summability and modulation estimates that remain stable under changes of
boundary coordinates and depend quantitatively on the local geometry.

In particular, it would be important to determine how curvature,
higher-order boundary jets, and the geometry of the glancing set enter
the multilinear interaction coefficients. Such information would
clarify whether the dyadic time-slab estimates and global continuation
mechanism established in the present model persist under perturbations
of the underlying geometry.

A successful extension would provide a pathway from the explicit
Friedlander model studied here toward nonlinear dispersive equations on
general manifolds with boundary. It would also distinguish the aspects
of the present argument that rely essentially on the explicit Airy
structure from those that reflect more general microlocal properties of
boundary propagation.

\subsection{Further Nonlinear Dispersive Models}

\label{subsec:further_models}

More broadly, the framework developed in this paper suggests that the
combination of boundary-adapted spectral analysis, Airy-mode multilinear
estimates, and high--low frequency iteration may be applicable to other
nonlinear dispersive equations for which translation-invariant Fourier
analysis is unavailable or insufficient.

Natural extensions include other nonlinear Schr\"odinger equations,
higher-order dispersive models, and wave equations exhibiting glancing
or caustic boundary interactions. For each such problem, the principal
analytical tasks are to identify an appropriate boundary spectral
representation, quantify the associated linear dispersive losses, and
determine whether the resulting multilinear interaction coefficients
satisfy estimates compatible with those losses.

The present work provides a model framework for this program. In
particular, it illustrates how the geometry of the boundary, the
continuous--discrete spectral decomposition, and the nonlinear
high--low iteration can be incorporated into a single global
well-posedness argument. Extending this mechanism to broader classes of
geometries and nonlinearities remains an important direction for future
research.

\end{document}